\documentclass[a4paper,11pt,twoside]{amsart}
\usepackage[latin1]{inputenc}
\usepackage[T1]{fontenc}

\usepackage{a4wide}

\usepackage{ifpdf}
\ifpdf\usepackage[pdftex]{hyperref}
\else\usepackage[hypertex]{hyperref}\fi

\usepackage{amsmath,amssymb,amsthm}

\theoremstyle{plain}
\newtheorem{thm}{Theorem}[section]
\newtheorem{prop}[thm]{Proposition}

\newtheorem{cor}[thm]{Corollary}

\theoremstyle{definition}
\newtheorem{defn}[thm]{Definition}

\theoremstyle{remark}
\newtheorem{rem}[thm]{Remark}

\theoremstyle{remark}

\DeclareMathOperator{\tr}{Tr}

\DeclareMathOperator{\im}{Im}

\let\dsp=\displaystyle 

\def\R{\mathbb R}
\def\C{\mathbb C}
\def\N{\mathbb N}

\begin{document}

\title[Spectral density and Sobolev inequalities]{Spectral density and
  Sobolev inequalities\\ for pure and mixed states}

\author{Michel Rumin}
\address{Laboratoire de Mathématiques d'Orsay\\
  CNRS et Université Paris Sud\\
  91405 Orsay Cedex\\ France}

\email{michel.rumin@math.u-psud.fr}

\date{\today}

\begin{abstract}
  We prove some general Sobolev--type and related inequalities for
  positive operators $A$ of given ultracontractive spectral decay
  $F(\lambda)= \|\chi_A(]0, \lambda])\|_{1,\infty}$, without assuming
  $e^{-tA}$ is submarkovian. These inequalities hold on functions, or
  pure states, as usual, but also on mixed states, or density
  operators in the quantum-mechanical sense. This provides universal
  bounds of Faber--Krahn type on domains $\Omega$, that apply to the
  whole Dirichlet spectrum distribution of $\Omega$, not only the
  first eigenvalue. Another application is given to relate the
  Novikov--Shubin numbers of coverings of finite simplical complexes
  to the vanishing of the torsion of the $\ell^{p,2}$-cohomology for
  some $p\geq 2$.
\end{abstract}



\thanks{Author supported in part by the French ANR-06-BLAN60154-01
  grant.}

\maketitle


\section{Introduction and main results}
\label{sec:introduction}

Let $A$ be a strictly positive self-adjoint operator on a measure
space $(X, \mu)$. Suppose moreover that the semigroup $e^{-tA}$ is
equicontinuous on $L^1(X)$ (submarkovian for instance). Then,
according to Varopoulos \cite{Varopoulos,Coulhon2}, a polynomial heat
decay
\begin{displaymath}
  \|e^{-tA}\|_{1,\infty} \leq C t^{-\alpha/2} \quad \mathrm{with} \quad
  \alpha > 2 \,,
\end{displaymath}
is equivalent to the Sobolev inequality
\begin{equation}
  \label{eq:1}
  \|f \|_p \leq C'
  \|A^{1/2}f\|_2 \quad \mathrm{for} \quad
  1/p = 1/2 - 1/\alpha.  
\end{equation}
This result applies in particular in the case $A$ is the Laplacian
acting \emph{on scalar functions} of a complete manifold, either in
the smooth or discrete graph setting.

\subsection{General Sobolev--Orlicz inequalities}
\label{sec:gener-sobol-orlicz}

The first purpose of this paper is to present short proofs of general
Sobolev--Orlicz inequalities that hold for positive self-adjoint
operators, without equicontinuity or polynomial decay assumption,
knowing either their heat decay, as above, or the ``ultracontractive
spectral decay'' $F(\lambda) = \|\Pi_\lambda\|_{1,\infty}$ of their
spectral projectors $\Pi_\lambda = \chi_A(]0, \lambda])$ on
$E_\lambda$. As will be seen, the interest for this former
$F(\lambda)$ mostly comes from geometric considerations. For instance
if $A$ is a scalar invariant operator over an unimodular group
$\Gamma$, then $F(\lambda) $ coincides with von Neumann's
$\Gamma$-dimension of $E_\lambda$, and thus $F$ represents the
non-zero spectral density function of $A$, see
Proposition~\ref{prop:density}. In the general setting the spectral
decay $F$ stays a right continuous increasing function as comes from
the identity
\begin{equation}
  \label{eq:2}
  \|P^*P\|_{1, \infty} = \|P\|_{1, 2}^2 = \sup_{\|f\|_1,
    \|g\|_1 \leq 1}| \langle Pf, Pg \rangle |.
\end{equation}
We first state the Sobolev--Orlicz inequality we shall prove on a
single function, or ``pure state'', as usual. In the sequel, if
$\varphi$ is a monotonic function, $\varphi^{-1}$ will denote its
right continuous inverse.

\begin{thm}
  \label{thm:H-Sobolev}
  Let $A$ be a positive self-adjoint operator on $(X, \mu)$ with
  ultracontractive spectral projections $\Pi_\lambda =
  \chi_A(]0,\lambda])$, i.e. $F(\lambda) = \|\Pi_\lambda\|_{1,\infty}
  < +\infty$.

  Suppose moreover that the Stieljes integral $\dsp G(\lambda) =
  \int_0^\lambda \frac{dF(u)}{u}$ converges.  Then any non zero $f \in
  L^2(X) \cap (\ker A)^\bot$ of energy $\mathcal{E}(f) = \langle Af,
  f\rangle_2$ satisfies
  \begin{equation}
    \label{eq:3}
    \int_X H \Bigl( \frac{|f(x)|^2}{4 \mathcal{E}(f)} \Bigr) d \mu \leq 1 \,, 
  \end{equation}
  where $H(y) = y \,G^{-1}(y)$.
\end{thm}

The heat version of this result has a similar statement (and proof).
\begin{thm}
  \label{thm:N-Sobolev}
  Let $A$ be a positive self-adjoint operator on $(X, \mu)$ such that
  $L(t) = \|e^{-tA} \Pi_V\|_{1,\infty}$ is finite, with $V= L^2(X)\cap
  (\ker A)^\bot$.
  
  Suppose moreover that $\dsp M(t) = \int_t^{+\infty} L(u) du < +
  \infty $. Then any non zero $f \in V$ of finite energy satisfies
  \begin{equation}
    \label{eq:4}
    \int_X N \Bigl( \frac{|f(x)|^2}{4 \mathcal{E}(f)} \Bigr) d \mu
    \leq \ln 2 \,,
  \end{equation}
  where $N(y) = y / M^{-1}(y)$
\end{thm}
Both results give (effective) Sobolev inequalities \eqref{eq:1} in the
polynomial decay case for $F$ or $L$. At first, we will see in
\eqref{eq:24} that the transform from $F$ to $G$ is increasing in
general, while $G$ to $H$ is decreasing. Therefore, if $F(\lambda)
\leq C \lambda^\alpha$ for $\alpha > 1$, then $G(\lambda) \leq C_1
\lambda^{\alpha-1}$ with $C_1 = \frac{C \alpha }{\alpha-1}$, and $H(y)
\geq C_1^{\frac{1}{1-\alpha}} y^{\frac{\alpha}{\alpha-1}}$. Hence
\eqref{eq:3} reads $\|f \|_{2\alpha/(\alpha-1)} \leq 2
C_1^{\frac{1}{2\alpha}} \|A^{1/2}f\|_2$.

Other related inequalities: generalised Moser, Nash and Faber--Krahn
inequalities are stated in Theorem~\ref{thm:Nash-Moser}.  Also an
application of Theorem~\ref{thm:H-Sobolev} is given below to the study
of $\ell^2$-cohomology of coverings of simplical complexes, but we
will first consider another issue.

\subsection{From pure to mixed states}
\label{sec:from-pure-mixed}

Namely we note that, from the quantum-mechanical viewpoint,
Theorems~\ref{thm:H-Sobolev} and \ref{thm:N-Sobolev} are inequalities
dealing with the density $|f(x)|^2$ of a single particle, or pure
state. Since we start from the knowledge of a strong ``collective
data'', related to the vector space $E_\lambda$, it is tentative to
look for a collective version of \eqref{eq:3}; that would handle many
functions simultaneously. A classical approach in statistical quantum
mechanics consists in replacing the orthogonal projection $\Pi_f$ on
$f$, by a mixed state, that is a positive linear combination of such
projections, or more generally by a positive operator $\rho$, see
e.g. \cite{Wikipedia}.

When dealing with a pure state, $|f(x)|^2$ interprets as the diagonal
value $K_{\rho}(x,x)$ of the kernel of $\rho =\Pi_f$. We need then to
extend this notion to general density operators $\rho$.  Moreover it
is also useful in geometry to consider operators acting on vector
valued, or even Hilbert valued functions. For instance, one may work
on differential forms of higher degrees. One can also consider
$\Gamma$-coverings $\widetilde M $ of compact manifolds $M$; in which
case one may set $X = \Gamma$ and use $L^2(\widetilde M) = L^2(\Gamma,
H)$ with $H= L^2 (\mathcal{F})$ for a fundamental domain
$\mathcal{F}$. To handle such cases, we will rely on the following
approach.

\begin{defn}
  \label{def:P-measure}
  Let $(X,\mu)$ be a $\sigma$-finite measure space, $H$ a separable
  Hilbert space and $P$ a positive operator acting on $\mathcal{H} =
  L^2(X,H)= L^2(X, \mu) \otimes H$. Then, given a measurable $\Omega
  \subset X$, the following trace
  \begin{equation}
    \label{eq:5}
    \nu_P(\Omega) = \tau(\chi_\Omega P \chi_\Omega) = \tau (P^{1/2}
    \chi_\Omega P^{1/2}) 
  \end{equation}
  defines an absolutely continuous measure on $X$ with respect to
  $\mu$.  Its Radon--Nikodym derivative $\dsp D \nu_P =\frac{d
    \nu_P}{d \mu} $ will be called the \emph{density function} of $P$,
  and
  \begin{displaymath}
    D(P) = \|D
    \nu_P\|_\infty = \sup_\Omega \frac{\nu_P(\Omega)}{\mu(\Omega)} 
  \end{displaymath}
  the \emph{density} of $P$.
\end{defn}

The following properties summarise the relationships between this
density, ultracontractivity and von Neumann $\Gamma$-trace.  We refer
for instance to \cite[\S 2]{Pansu1} for an introduction on this last
subject.

\begin{prop}
  \label{prop:density}
  $\bullet$ With $P$ and $\mathcal{H}$ as above, one has for any
  Hilbert basis $(e_i)$ of $\mathcal{H}$
  \begin{equation}
    \label{eq:6}
    \nu_P(\Omega) = \int_\Omega \sum_i \|(P^{1/2}e_i)(x)\|_H^2 d\mu(x) \quad
    \mathrm{and} \quad D\nu_P(x) = \sum_i \|(P^{1/2}e_i)(x)\|_H^2 \,.
  \end{equation}

  $\bullet$ If $H$ is finite dimensional, a positive $P$ is
  ultracontractive if and only if it has a bounded density and
  \begin{equation}
    \label{eq:7}
    \|P\|_{1,\infty}  \leq D(P) \leq (\dim H) \|P\|_{1,\infty} \,.
  \end{equation}

  $\bullet$ When $X$ is a locally compact group $\Gamma$ with its Haar
  measure, and $P$ a translation invariant operator, then the density
  function of $P$ is a constant number, so that $\nu_P(\Omega) = D(P)
  \mu(\Omega)$. Moreover, it coincides with von Neumann $\Gamma$-trace
  of $P$ when $\Gamma$ is discrete. Namely if $K_P$ denotes the kernel
  of $P$, one has in this case
  \begin{equation}
    \label{eq:8}
    D(P) = \tau_H(K_P(e,e)) =  \tau_\Gamma (P) \,.
  \end{equation}
\end{prop}
In this setting, the mixed state version of
Theorem~\ref{thm:H-Sobolev} is the following.

\begin{thm}
  \label{thm:rho-Sobolev}
  Let $(X,\mu)$ be a measure space, $H$ an Hilbert space, and $A$ a
  positive self-adjoint operator on $\mathcal{H} = L^2(X,H)$.

  Suppose that the spectral projections $\Pi_\lambda= \chi_A (]0,
  \lambda])$ have finite density $F(\lambda) = D(\Pi_\lambda) $, and
  that $G(\lambda) = \int_0^\lambda \frac{d F(u)}{u}$ converges. Let
  $\rho$ be a positive operator such that $\rho = 0$ on $\ker A$. Then
  \begin{equation}
    \label{eq:9}
    \int_X G^{-1} \Bigl( \frac{D \nu_\rho(x)}{4 \|\rho^{1/2} A \rho^{1/2}
      \|_{2,2}} \Bigr) d \nu_\rho
    \leq 4 \mathcal{E}(\rho)\,,
  \end{equation}
  where
  \begin{displaymath}
    \mathcal{E}(\rho) = \tau(\rho^{1/2} A
    \rho^{1/2}) = \tau(A^{1/2} \rho A^{1/2}) \ (= \tau(A\rho) \quad
    \mathrm{if\ finite})\,,
  \end{displaymath}
  and $\|\rho^{1/2} A \rho^{1/2} \|_{2,2} $ is the $L^2-L^2$ norm of
  $\rho^{1/2} A \rho^{1/2}$.
\end{thm}

The heat version of the Sobolev inequality Theorem~\ref{thm:N-Sobolev}
has also a mixed state or $\rho$-version, replacing above $F$ by $L
(t)= D(e^{-tA})$, and $G$ by $\widetilde M$ where $\widetilde{ M}
(1/t) = \int_t^{+\infty } L(s) ds$.

\medskip

To illustrate Theorem~\ref{thm:rho-Sobolev}, suppose again that $F$
has a polynomial growth
\begin{equation}
  \label{eq:10}
  F(\lambda) \leq C \lambda^\alpha \quad \mathrm{for\ some} \quad
  \alpha > 1 \,. 
\end{equation}
As examples of mixed states, we take $\rho$ to be the projection onto
a $N$-dimensional linear space $V$ of functions such that $\mathcal{E}
(f) = \langle Af, f \rangle_2 \leq \lambda \|f\|_2^2$ for all $f \in
V$. Then \eqref{eq:6} yields
\begin{displaymath}
  D \nu_\rho (x ) = \sum_{i=1}^N \|f_i(x)\|_H^2\,,
\end{displaymath}
for any orthonormal basis $(f_i)$ of $V$.  Since as previously
$G(\lambda) \leq C_1 \lambda^{\alpha-1} $ with $C_1=
\frac{C\alpha}{\alpha-1} $, the $\rho$-Sobolev inequality \eqref{eq:9}
provides
\begin{equation}
  \label{eq:11}
  \int_X \bigl( \sum_{i=1}^N \|f_i(x)\|_H
  ^2
  \bigr)^{\frac{\alpha}{\alpha-1}} d \mu \leq C_2
  \lambda^{\frac{1}{\alpha-1}} \sum_{i=1}^N \mathcal{E}(f_i) \leq C_2
  N \lambda^{\frac{\alpha}{\alpha-1}} \,,
\end{equation}
with $C_2 = 4^{\frac{\alpha}{\alpha -1}} C_1^{\frac{1}{\alpha-1}}$.

\smallskip

If $\Omega$ is any domain of finite measure, these inequalities give
back some control of the spectral distribution of the Dirichlet
spectrum of $\mathcal{E}$ on $\Omega$. Namely, if the states $f_i$ are
supported in $\Omega$, then by Jensen, one finds that
\begin{displaymath}
  (N/ \mu (\Omega))^{\alpha/\alpha-1} =  \Bigl( \int_\Omega \sum_{i=1}^N
  \|f_i(x)\|_H^2 d\mu / \mu(\Omega) \Bigr)^{\alpha/\alpha-1} \leq C_2
  (N/\mu(\Omega)) 
  \lambda^{\frac{\alpha}{\alpha-1}} \,,
\end{displaymath}
yielding
\begin{equation}
  \label{eq:12}
  \frac{\dim V}{\mu(\Omega)} \leq \frac{4^{\alpha} \alpha}{\alpha-1}
  C \lambda^\alpha = C_3 \lambda^\alpha, 
\end{equation}
for any $V$ supported in $\Omega$ and such that $\mathcal{E}(f)\leq
\lambda \|f\|_2^2$ on $V$.

In the case $\dim V= 1$, i.e. the pure state case, this interprets as
a Faber--Krahn inequality, on the lower bound for the Dirichlet
spectrum of $\mathcal{E}$ on $\Omega$. While using mixed states, one
gets actually the control of the \emph{whole} geometrical spectral
repartition function
\begin{equation}
  \label{eq:13}
  F_\Omega^{\dim}(\lambda) =
  \sup_{V \subset \subset \Omega} \{\dim V \mid \mathcal{E} \leq
  \lambda \ \mathrm{on\ }V\} \leq 
  \mu(\Omega) C_3 \lambda^\alpha \,.
\end{equation}
Note that if $X$ itself has a finite measure, this control is coherent
with the starting hypothesis \eqref{eq:10} on the spectral density $F$
of $E_\lambda$, since in finite measure
\begin{equation}
  \label{eq:14}
  F_X^{\dim} = \dim E_\lambda \leq \mu(X) F \,,
\end{equation}
as comes from Definition~\ref{def:P-measure} and
\begin{equation}
  \label{eq:15}
  \dim V = \tau(\Pi_V) = \nu_{\Pi_V}(X)  \leq
  \mu(X) D( \Pi_V) \,.
\end{equation}
We note also that for invariant operators and spaces on finite groups,
\eqref{eq:14} and \eqref{eq:15} are equalities by
Proposition~\ref{prop:density}. Hence the ``$\rho$-Sobolev''
inequalities \eqref{eq:4} capture back the bound on the spectral
density $F$, at least in this simple setting.

\subsection{Moser inequalities for mixed states and global Faber--Krahn
  inequalities}
\label{sec:Moser-Faber-Krahn}

As done in the polynomial case above, one can show that the previous
$\rho$-Sobolev inequalities also imply Faber--Krahn and spectral
inequalities for other spectral densities, see
Proposition~\ref{prop:Sobolev-Faber-Krahn}. However this approach
assumes some thinness of the spectrum, as required by the convergence
of $G$ or $M$. As Faber--Krahn inequalities make sense for thick
spectrum, we now present another viewpoint.

The starting point is a generalised Moser inequality for mixed
state. We will give an integral version, as above, but also a discrete
one, associated to a partition of $X$ into $\bigsqcup_i \Omega_i$.  We
state the result under two close sets of hypothesis; depending whether
one remove the kernel of $A$ from the spectral density and the states,
as needed in the previous approach, or not. The integral version is
the following.

\begin{thm}
  \label{thm:F-Sobolev} Let $A$ be a positive self-adjoint operator on
  $L^2(X,H)$ and $\rho$ be a non-zero positive operator. Suppose
  either
  \begin{itemize}
  \item $\rho =0$ on $\ker A$ and $F_x(\lambda) = D
    \nu_{\Pi_{]0,\lambda]}}(x)$ denotes the density at $x$ of
    $\Pi_{]0,\lambda]} = \chi_A(]0, \lambda])$,
  \item or $\rho$ is whatever and $F_x(\lambda) = D
    \nu_{\Pi_{[0,\lambda]}}(x)$ is the density of $\Pi_{[0,\lambda]} =
    \chi_A([0, \lambda])$.
  \end{itemize}
  Then the following generalised Moser inequality holds for the
  density operator $\rho$
  \begin{equation}
    \label{eq:16}
    \int_X F_x^{-1}\Bigl( \frac{D \nu_\rho(x)}{4 \|\rho\|_{2,2}}\Bigr) d \nu_\rho
    \leq 4  \mathcal{E}(\rho)\,.
  \end{equation}
\end{thm}

The corresponding result for partitions writes:

\begin{thm}
  \label{thm:rho-Moser-partition}
  Let $A$ as above and $X = \bigsqcup_I \Omega_i$ be a discrete
  measurable partition of $X$.

  Consider the $\Pi_\lambda$-measures of $\Omega_i$
  \begin{displaymath}
    F_{\Omega_i}(\lambda) = \tau (\chi_{\Omega_i} \Pi_\lambda
    \chi_{\Omega_i}) = \nu_{\Pi_\lambda} (\Omega_i) \,,
  \end{displaymath}
  where $\Pi_\lambda$ denotes either $\Pi_{]0,\lambda]}$ or
  $\Pi_{[0,\lambda]}$. Let $\rho$ be a non-zero positive operator,
  with the additional assumption that $\rho=0$ on $\ker A$ if using
  $\Pi_{]0,\lambda]}$.
  
  In this discrete setting the generalised Moser inequality writes
  \begin{equation}
    \label{eq:17}
    \sum_i F_{\Omega_i}^{-1} \Bigl( \frac{\nu_\rho (\Omega_i)}{4
      \|\rho\|_{2,2}} \Bigr) \nu_\rho (\Omega_i) \leq 4 \mathcal{E}(\rho)\,.
  \end{equation}
\end{thm}
This yields the following ``global'' Faber--Krahn inequality.
\begin{cor}
  \label{cor:Faber-Krahn}
  With the same notations as above, suppose moreover that $\rho$ is
  supported in a domain $\Omega$ of finite measure and has finite
  energy $\mathcal{E}(\rho)$, then
  \begin{equation}
    \label{eq:18}
    \frac{\tau(\rho)}{4 \|\rho\|_{2,2}} \leq F_\Omega \bigl( 4 \langle A 
    \rangle_\rho \bigr) \leq \mu(\Omega) F \bigl( 4 \langle A 
    \rangle_\rho \bigr) \,.
  \end{equation}
  where $\dsp \langle A \rangle_\rho =
  \frac{\mathcal{E}(\rho)}{\tau(\rho)} =\frac{\tau(A \rho)}{ \tau
    (\rho)} $ is the expectation value of $A$ with respect to $\rho$.

  $\bullet$ In particular the whole Dirichlet spectrum of
  $\mathcal{E}$ on $\Omega$ is controlled for all $\lambda$ by
  \begin{equation}
    \label{eq:19}
    F_\Omega^{\dim} (\lambda) \leq 4 \mu(\Omega) F(4 \lambda) \,,
  \end{equation}
  where as in~\eqref{eq:13}, we define $F_\Omega^{\dim} (\lambda) =
  \sup \{\dim V \mid \mathrm{supp}( V) \subset \Omega \ \mathrm{and}\
  \mathcal{E} \leq \lambda \ \mathrm{on\ }V\}$.
\end{cor}

Thus \eqref{eq:19} extends uniformly, whatever $F$, the bounds
\eqref{eq:13} obtained previously from the $\rho$-Sobolev
inequalities. In some sense it means that the spectral density of a
confined system is controlled by the spectral density of the free
system, up to universal multiplicative constants in volume and
energy. In the case of the Laplacian in $\R^n$, or more general
Schr\"odinger operators, such inequalities have been proved
independently by Cwikel, Lieb and Rosenb{\l}jum, see \cite{Lieb}.

We note that except for these constants $4$, the formula \eqref{eq:19}
looks quite sharp in general. Indeed, as recalled above, one has
$F_\Gamma^{\dim} = \mu(\Gamma) F$ when $A$ is an invariant operator on
a finite group $\Gamma=X$. Hence in general
\begin{equation}
  \label{eq:20}
  F_\Omega^{\dim} (\lambda)\leq \mu (\Omega) F(\lambda)
\end{equation}
is certainly an ideal bound for inequalities like \eqref{eq:19}.

\subsection{Sobolev inequalities and $\ell^2$-cohomology}
\label{sec:sobol-ineq-ell2}

We conclude with an application to geometric analysis of the pure
state case of Sobolev inequalities in Theorem~\ref{thm:H-Sobolev} or
\ref{thm:N-Sobolev}. As they are not restricted to Markovian
operators, these results apply in the following setting. Let $K$ be a
finite simplicial complex and $X \rightarrow K = X / \Gamma$ some
covering. One considers on $X$ the complex of $\ell^2$ $k$-cochains
with the discrete coboundary
\begin{displaymath}
  d_k : \ell^2 X^k \rightarrow \ell^2 X^{k+1} 
\end{displaymath}
dual to the usual boundary $\partial$ of simplexes, see e.g. \cite[\S
3]{Pansu1}.

Its $\ell^2$-cohomology $H_2^{k+1} = \ker d_{k+1} / \im d_k$ splits in
two components :
\begin{itemize}
\item the reduced part $\overline H_2^{k+1} = \ker d_{k+1} /
  \overline{ \im d_k}$, isomorphic to $\ell^2$-harmonic cochains
  $\mathcal{H}_2^{k+1} = \ker d_{k+1} \cap \ker d_k^*$,
\item and the torsion $T_2^{k+1} = \overline{ \im d_k} / \im d_k$.
\end{itemize}
Although this torsion is not a normed space, one can study it by
``measuring'' the unboundedness of $d_k^{-1} $ on $\im d_k$.
We will consider here two different means.

\smallskip - A first one is inspired by $\ell^{p,q}$-cohomology. One
enlarges the space $\ell^2 X^k $ to $\ell^p X^k$ for $p \geq 2$, and
asks whether, for $p$ large enough, one has
\begin{equation}
  \label{eq:21}
  \overline{d_k
    (\ell^2 X^k)}^{\ell^2}  \subset d_k(\ell^p X^k  )\,, 
\end{equation} 
This is satisfied in case the following Sobolev identity holds
\begin{equation}
  \label{eq:22}
  \exists C \quad \mathrm{such\ that }\quad  \|\alpha\|_p \leq C \|d_k
  \alpha \|_2 \quad \mathrm{for\ all}\ \alpha 
  \in (\ker d_k)^\bot  \subset  \ell^2\,. 
\end{equation}
The geometric interest of the rougher formulation \eqref{eq:21} lies
in its stability under the change of $X$ into other bounded homotopy
equivalent spaces, as stated in
Proposition~\ref{prop:invariance-torsion}. Moreover if $\overline
H_2^{k+1}(X)$ vanishes, then \eqref{eq:21} is equivalent to the
vanishing of the torsion of the $\ell^{p,2}$-cohomology of $X$, as
will be seen in Section~\ref{sec:spectr-dens-cohom}.

\smallskip

- The second approach is spectral and relies on the von Neumann
$\Gamma$-trace of $\Pi_{]0,\lambda]}$, i.e to the spectral density by
Proposition~\ref{prop:density}. Consider the $\Gamma$-invariant
self-adjoint $A = d_k^* d_k $ acting on $(\ker d_k)^\bot$ and the
spectral density $F_{\Gamma, k}(\lambda) = \tau_\Gamma(\Pi_{]0,
  \lambda]})$.  This function vanishes near zero if and only if zero
is isolated in the spectrum of $A$, which is equivalent to the
vanishing of the torsion $T_2^{k+1}$. The asymptotic behaviour of
$F_{\Gamma,k}(\lambda)$ when $\lambda \searrow 0$ has a geometric
interest in general since, given $\Gamma$, it is an homotopy invariant
of the quotient space $K$, as shown by Efremov, Gromov and Shubin in
\cite{Efremov,Gromov-Shubin,Gromov-Shubin-erratum}.

One can compare these two notions in the spirit of Varopoulos result
\eqref{eq:1} on functions. In the case of polynomial decay one
obtains.
\begin{thm}
  \label{thm:1.4} Let $K$ be a finite simplicial space and $ X
  \rightarrow K = X / \Gamma$ a covering. Let $F_{\Gamma,k}(\lambda)=
  \dim_\Gamma E_\lambda $ denotes the spectral density function of $A
  = d_k^* d_k$ on $(\ker d_k)^\bot$.

  If $F_{\Gamma,k}(\lambda) \leq C \lambda^{\alpha/2}$ for some
  $\alpha > 2 $, then the Sobolev inequality \eqref{eq:22}, and the
  inclusion \eqref{eq:21}, hold for $1/p \leq 1/2 - 1/\alpha$.
\end{thm}

If moreover the reduced $\ell^2$-cohomology $\overline{H}^{k+1}_2(X)$
vanishes, this implies the vanishing of the $\ell^{p,2}$-torsion of
$X$, as stated in Corollary~\ref{cor:lp2}.

Other spectral decays than polynomial can be handled with
Theorem~\ref{thm:H-Sobolev}, leading then to a bounded inverse of
$d_k$ from $\im d_k \cap \ell^2$ into a more general Orlicz space
given by $H$.

The author thanks Pierre Pansu and Michel Ledoux for their comments on
this work.

\section{Proof of the pure state inequalities}
\label{sec:proofs-pure-state}

The first step towards Theorems~\ref{thm:H-Sobolev} to
\ref{thm:N-Sobolev} is to consider the ultracontrativity of the
auxiliary operators $A^{-1} \Pi_\lambda$ and $A^{-1}e^{-tA}\Pi_V$.
\begin{prop}
  \label{prop:2.1} $\bullet$ Let $A$, $F$ and $G$ be given as in
  Theorem~\ref{thm:H-Sobolev}. Then $A^{-1} \Pi_\lambda$ is
  ultracontractive with
  \begin{equation}
    \label{eq:23}
    \|A^{-1} \Pi_\lambda\|_{1, \infty}  \leq G(\lambda) =
    \int_0^\lambda \frac{dF(u)}{u} \,.
  \end{equation}
  $\bullet$ Let $A$, $L$ and $M$ be given as in
  Theorem~\ref{thm:N-Sobolev}.  Then $A^{-1} e^{-tA}\Pi_V$ is
  ultracontractive with
  \begin{equation}
    \label{eq:24}
    \|A^{-1} e^{-tA} \Pi_V\|_{1,\infty}  \leq M(t) = \int_t^{+\infty}
    L(s) ds\,.
  \end{equation}
\end{prop}
\begin{proof}
  $\bullet$ The spectral calculus gives
  \begin{equation}
    \label{eq:25}
    A^{-1} (\Pi_\lambda - \Pi_\varepsilon)  = \int_{]\varepsilon,
      \lambda]} u^{-1} d \Pi_u  = \lambda^{-1} \Pi_\lambda -
    \varepsilon^{-1} \Pi_\varepsilon
    + \int_{]\varepsilon, \lambda]}
    u^{-2}\Pi_u du\,,   
  \end{equation}
  thus taking norms, one obtains
  \begin{align*}
    \|A^{-1} (\Pi_\lambda - \Pi_\varepsilon)\|_{1,\infty} & \leq
    \lambda^{-1} F(\lambda) + \varepsilon^{-1} F(\varepsilon) +
    \int_{]\varepsilon, \lambda]}
    u^{-2} F(u) du \\
    & = G(\lambda) -G (\varepsilon) + 2 \varepsilon^{-1}
    F(\varepsilon) \,.
  \end{align*}
  Now by finiteness of $G$, one has $\|\Pi_\varepsilon / \varepsilon
  \|_{1, \infty} = F(\varepsilon)/ \varepsilon \leq G(\varepsilon)
  \rightarrow 0$ when $\varepsilon \searrow 0$, hence by \eqref{eq:2}
  \begin{align*}
    \|A^{-1}\Pi_\lambda\|_{1,\infty}
    & = \|\Pi_\lambda A^{-1/2}\Pi_\lambda\|_{1,2}^2 \\
    & = \lim_{\varepsilon \rightarrow 0} \| (\Pi_\lambda -
    \Pi_\varepsilon)A^{-1/2} \Pi_\lambda\|_{1,2}^2 \quad \textrm{by
      Beppo-Levi,}\\
    & = \lim_{\varepsilon \rightarrow 0} \|A^{-1} (\Pi_\lambda -
    \Pi_\varepsilon)\|_{1,\infty} \leq G(\lambda)\,.
  \end{align*}
  We note that we also have
  \begin{equation}
    \label{eq:26}
    G(\lambda) = \lambda^{-1} F(\lambda) + \int_0^\lambda u^{-2} F(u) du\,,
  \end{equation}
  which shows the useful monotonicity of the transform from $F$ to $G$
  and $H$.

  $\bullet$ The heat case \eqref{eq:24} is clear since $A^{-1} e^{-tA}
  \Pi_V = \int_t^{+\infty} e^{-sA}\Pi_V ds$ by the spectral calculus.
\end{proof}

The sequel of the proofs of Theorems~\ref{thm:H-Sobolev}
and~\ref{thm:N-Sobolev} relies on a classical technique from real
interpolation theory, as used for instance in the elementary proof of
the $L^2-L^p$ Sobolev inequality in $\R^n$ given by Chemin and Xu in
\cite{Chemin-Xu}. This consists here in estimating each level set
$\{x, |f(x)| > y\}$ by using an appropriate spectral splitting of $f
\in V$ into
\begin{equation}
  \label{eq:27}
  f= \chi_A(]0, \lambda]) f + \chi_A(]\lambda, +\infty[) f =
  \Pi_\lambda f + \Pi_{> \lambda}f\,. 
\end{equation}

\subsection{Proof of Theorem~\ref{thm:H-Sobolev}}
\label{sec:proof-H-Sobolev}
By \eqref{eq:2} and \eqref{eq:23} one has $\|A^{-1/2} \Pi_\lambda
\|_{2, \infty}^2 \leq G(\lambda)$, hence
\begin{equation}
  \label{eq:28}
  \|\Pi_\lambda f\|_\infty^2 \leq G(\lambda) \|A^{1/2} f\|^2_2 =
  G(\lambda) \mathcal{E}(f)\,. 
\end{equation}
Then suppose that $|f(x)| \geq y$, with $y^2 = 4 G(\lambda)
\mathcal{E}(f)$. As $|\Pi_\lambda f(x)| \leq y/2$ by \eqref{eq:28},
one has necessarily by \eqref{eq:27} that $|\Pi_{> \lambda} f(x)| \geq
y /2 \geq |\Pi_\lambda f(x)|$ and finally
\begin{equation}
  \label{eq:29}
  |f(x) |^2 \leq 4 |\Pi_{> \lambda} f(x)|^2 \quad \mathrm{on} \quad
  \bigl\{ x \in X \mid |f(x)|^2 \geq 4 G(\lambda)
  \mathcal{E}(f) \bigr\} \,. 
\end{equation}
Hence a first integration in $x$ gives,
\begin{displaymath}
  \int_{\{x \, , \, |f(x)|^2 \geq 4 \mathcal{E}(f) G(\lambda)\}}
  |f(x)|^2 d\mu \leq 4 \|\Pi_{> \lambda} f\|_2^2 \,,
\end{displaymath}
and a second integration in $\lambda$,
\begin{displaymath}
  \int_X \frac{|f(x)|^2}{4 \mathcal{E}(f)} G^{-1}\Bigl( \frac{|f(x)|^2}{4
    \mathcal{E}(f)} \Bigr) d \mu (x) \leq \int_0^{+\infty} \frac{\|\Pi_{> \lambda}
    f\|_2^2 }{  \mathcal{E}(f)} d \lambda \,, 
\end{displaymath}
where $G^{-1}(y) = \sup\{\lambda \mid G(\lambda) \leq y \}$. At last
the spectral calculus provides
\begin{align*}
  \int_0^{+\infty} \|\Pi_{> \lambda} f\|_2^2 \, d \lambda & =
  \int_0^{+\infty } \int_\lambda^{+\infty}
  \langle d \Pi_\mu f, f\rangle  \\
  & = \int_0^{+\infty} \mu \, \langle d \Pi_\mu f, f\rangle = \langle
  Af, f \rangle = \mathcal{E}(f) \,,
\end{align*}
proving Theorem~\ref{thm:H-Sobolev}.

\subsection{Proof of Theorem~\ref{thm:N-Sobolev}}
\label{sec:proof-N-Sobolev}
We follow the same lines as above. First by \eqref{eq:2} and
\eqref{eq:24} one has for $f\in V$
\begin{displaymath}
  \|e^{-tA/2} f\|_\infty \leq M(t) \mathcal{E}(f)\,,
\end{displaymath}
leading to
\begin{equation}
  \label{eq:30}
  |f(x) |^2 \leq 4 |(1- e^{-tA/2}) f(x)|^2 \quad \mathrm{on} \quad
  \bigl\{ x \in X \mid |f(x)|^2 \geq 4 M(t) \mathcal{E}(f)
  \bigr\}\,.
\end{equation}
Then integrations in $x$ and $dt/t^2$ give
\begin{displaymath}
  \int_X \frac{|f(x)|^2}{4 \mathcal{E}(f)} / M^{-1}\Bigl( \frac{|f(x)|^2}{4
    \mathcal{E}(f)} \Bigr) \, d \mu (x) \leq \frac{1}{\mathcal{E}(f)}
  \int_0^{+\infty} \|(1- 
  e^{-tA/2}) f\|_2^2 \, \frac{dt}{t^2} \,, 
\end{displaymath}
where now $M^{-1}(y)= \inf \{t \mid M(t) \leq y\}$ for the decreasing
$M$. The right integral is computed by spectral calculus
\begin{align*}
  \int_0^{+\infty} \|(1- e^{-tA/2}) f\|_2^2 \, \frac{dt}{t^2} & =
  \int_0^{+\infty} \int_0^{+\infty} (1- e^{-t\lambda/2})^2 \langle d
  \Pi_\lambda f, f
  \rangle \, \frac{dt}{t^2}  \\
  & = \int_0^{+\infty} \Bigl( \int_0^{+\infty} \frac{(1- e^{-u})^2 }{2
    u^2} du \Bigr) \lambda \langle d \Pi_\lambda f, f \rangle \\
  & = I \mathcal{E}(f) \,,
\end{align*}
where $\dsp 2 I = \int_0^{+\infty} \frac{(1- e^{-u})^2}{u^2}du = 2 \ln
2$ as seen developing $\dsp I_\varepsilon = \int_\varepsilon^{+\infty}
\frac{(1- e^{-u})^2}{u^2}du$ when $\varepsilon \searrow 0$.

\subsection{Related inequalities}
\label{sec:related-inequalities}

Using the same technique as above one can also show some generalised
Moser, Nash and Faber--Krahn inequalities for functions. The shape of
the Moser inequality resembles to the ``F-Sobolev'' inequality
introduced by Wang in~\cite{Wang} for some Schr\"odinger operators.
\begin{thm}
  \label{thm:Nash-Moser} Let $A$ be a positive self-adjoint operator
  on $(X, \mu)$. Suppose either
  \begin{itemize}
  \item $f$ is a non-zero function in $V= L^2(X)\cap (\ker A)^\bot$
    and $F(\lambda)$ denotes $\|\Pi_{]0, \lambda]}\|_{1, \infty}$ as
    above,
  \item or $f$ is any non-zero function in $L^2(X)$, and $F(\lambda) =
    \|\Pi_{[0, \lambda]}\|_{1, \infty}$\,.
  \end{itemize}
  $\bullet$ Then the following generalised $L^2$ Moser inequality
  holds
  \begin{equation}
    \label{eq:31}
    \int_X |f(x)|^2 F^{-1}\Bigl( \frac{|f(x)|^2}{4 \|f\|_2^2}\Bigr) d \mu\leq 4
    \mathcal{E}(f)\,, 
  \end{equation}
  and also
  \begin{equation}
    \label{eq:32}
    \int_X |f(x)|^2 F^{-1}\Bigl( \frac{|f(x)|}{2 \|f\|_1}\Bigr) d \mu\leq 4
    \mathcal{E}(f)\,. 
  \end{equation}
  $\bullet$ Both inequalities imply a Nash--type inequality (with
  weaker constants starting from \eqref{eq:31})
  \begin{equation}
    \label{eq:33}
    \|f\|_2^2 F^{-1} \Bigl(\frac{ \|f\|_2^2}{4 \|f\|_1^2}\Bigr) \leq
    8 \mathcal{E}(f)\,.
  \end{equation}
  $\bullet$ In particular if $f$ is supported in a domain $\Omega$ of
  finite measure and has finite energy, the following Faber--Krahn
  inequality, or ``uncertainty principle'', is satisfied
  \begin{equation}
    \label{eq:34}
    4 \mu(\Omega) F \Bigl(\frac{8
      \mathcal{E}(f)}{\|f\|_2^2}\Bigr) \geq 1\,.
  \end{equation}
\end{thm}

Here one compares levels of $f$ to $\|f\|_2$ or $\|f\|_1$ instead of
$\mathcal{E}(f)$. This does not rely on Proposition~\ref{prop:2.1},
and one can work either with $f \in (\ker A)^\bot$ and $F(\lambda)= \|
\Pi_{]0,\lambda]} \|_{1,\infty}$, as before, or with a general $f \in
L^2(X)$ and $F(\lambda) = \|\Pi_{[0,\lambda]}\|_{1,\infty}$.  In any
case, starting from \eqref{eq:27} one gets
\begin{equation}
  \label{eq:35}
  \begin{gathered}
    |f(x) |^2 \leq 4 |\Pi_{> \lambda} f(x)|^2 \quad \mathrm{on} \\
    \bigl\{ x \in X \mid |f(x)|^2 \geq 4 F(\lambda) \|f\|_2^2 \bigr\}
    \quad \mathrm{or} \quad \bigl\{ x \in X \mid |f(x)| \geq 2
    F(\lambda) \|f\|_1 \bigr\} \,.
  \end{gathered}
\end{equation}
This yields the generalised Moser inequalities \eqref{eq:31} and
\eqref{eq:32} by integrations as in Theorem~\ref{thm:H-Sobolev}.

Note that in the case one works without restriction on $f$ and
$F(\lambda) = \|\Pi_{[0, \lambda]}\|_{1,+\infty}$, one has to complete
the definition of $F^{-1}$ by setting
\begin{equation}
  \label{eq:36}
  F^{-1}(y) = 
  \begin{cases}
    0 & \mathrm{if} \quad  y < F(0) \,,\\
    \sup \{\lambda \mid F(\lambda) \leq y \} & \mathrm{elsewhere}\,.
  \end{cases}
\end{equation}
This means that the inequalities \eqref{eq:31} and \eqref{eq:32} cut
off small values of $f$ in that case.

To deduce the Nash--type inequality \eqref{eq:33}, we argue as in
\cite[p. 97]{Coulhon-Grigorian-Levin}. Observe that for all
non-negative $s$ and $t$ one has
\begin{equation}
  \label{eq:37}
  st  \leq s F(s) + t F^{-1}(t)\,.
\end{equation}
Applying to $t= \frac{|f(x)|}{2 \|f\|_1}$ gives
\begin{displaymath}
  s \frac{|f(x)|}{2 \|f\|_1} - s F(s) \leq \frac{|f(x)|}{2 \|f\|_1}
  F^{-1}\Bigl( \frac{|f(x)|}{2 \|f\|_1}\Bigr) \,.
\end{displaymath}
By integration against the measure $|f(x)| d \mu$ and using
\eqref{eq:32}, this yields
\begin{displaymath}
  s \frac{\|f\|_2^2}{2\|f\|_1} - s F(s) \|f\|_1 \leq \int_X
  \frac{|f(x)|^2}{2 \|f\|_1} F^{-1}\Bigl( \frac{|f(x)|}{2 \|f\|_1}
  \Bigr) d \mu \leq \frac{2\mathcal{E}(f)}{\|f\|_1}\,.
\end{displaymath}
This provides \eqref{eq:33} using
\begin{displaymath}
  s \nearrow F^{-1}\Bigl(\frac{\|f\|_2^2}{4\|f\|_1^2} \Bigr) =
  \sup \Bigl\{s \mid F(s) \leq \frac{
    \|f\|_2^2}{4\|f\|_1^2} \Bigr\} \,.
\end{displaymath}

One can proceed similarly starting from the $L^2$ Nash--type
inequality \eqref{eq:31} instead of \eqref{eq:32}. One replaces
\eqref{eq:37} by
\begin{displaymath}
  st \leq s \sqrt{ F (s)} + t F^{-1}(t^2) 
\end{displaymath}
with $t = \frac{|f(x)|}{2 \|f\|_2}$. Integrating against $|f(x)| d
\mu$ and using $s \nearrow F^{-1}(\frac{\varepsilon^2
  \|f\|_2^2}{\|f\|_1^2})$ yields
\begin{displaymath}
  (\frac{1}{2} - \varepsilon)\|f\|_2^2  F^{-1}\Bigl(\frac{\varepsilon^2
    \|f\|_2^2}{\|f\|_1^2} \Bigr) \leq 2\mathcal{E}(f)\,,  
\end{displaymath}
which is similar to \eqref{eq:33}, but with weaker constants.

\smallskip

When is $f$ is supported in a domain $\Omega$ of finite measure, one
has $\|f\|_1^2 \leq \mu(\Omega) \|f\|_2^2$, and thus \eqref{eq:33}
implies that
\begin{displaymath}
  F^{-1} \Bigl( \frac{1}{4\mu
    (\Omega)} \Bigr) \leq \frac{8\mathcal{E}(f)}{\|f\|_2^2}\,.  
\end{displaymath}
If $\mathcal{E}(f)$ is finite, this leads to the Faber--Krahn
inequality \eqref{eq:34} since by right continuity of $F$ and
\eqref{eq:36}, one has $F(F^{-1}(\lambda)) \geq \lambda$ when
$F^{-1}(\lambda)$ is finite.

\subsection{Remark}
\label{sec:remark-1}
In the previous proofs, it appears clearly that the proposed controls
of ultracontractive norms of spectral or heat decay are much stronger
than the Sobolev or Nash-Moser inequalities deduced. Indeed these
inequalities are twice integrated versions, in space and frequency, of
the ``local'' inequalities \eqref{eq:29}, \eqref{eq:30} and
\eqref{eq:35}, that come directly from the ultracontractive controls.
Therefore it seems hopeless to get the converse statements in general.

However, we recall that one can get back from Sobolev or Nash
inequalities to the heat decay, in the case the heat is
\emph{equicontinuous} on $L^1$ or $L^\infty$. This is due to
Varopoulos in \cite{Varopoulos} for the polynomial case, and Coulhon
in \cite{Coulhon2} for more general decays.  This strong
equicontinuity hypothesis holds for the Laplacian on scalar functions,
as comes for instance from the maximum principle, but unfortunately
only in a positive curvature setting for Hodge-de Rham Laplacians
acting on forms of higher degree.

\section{Proof of the mixed state inequalities}
\label{sec:proof-mixed-state}

\subsection{Proof of $\rho$-Sobolev}
\label{sec:proof-rho-Sobolev}

The proof of the mixed state version of Sobolev inequality follows the
same lines as the pure state one.

The first step is adapted to use the density of operators, as given in
Definition~\ref{def:P-measure}, instead of their ultracontractive
norm.
\begin{prop}
  \label{prop:rho-G} $\bullet$ Let $A$, $F$ and $G$ be given as in
  Theorem~\ref{thm:rho-Sobolev}. Then $A^{-1} \Pi_\lambda$ has a
  finite density and
  \begin{equation}
    \label{eq:38}
    D(A^{-1} \Pi_\lambda)  \leq G(\lambda) =
    \int_0^\lambda \frac{dF(u)}{u} \,.
  \end{equation}
  $\bullet$ Suppose $L(t)= D(e^{-tA})$ and $M(t) = \int_s^{+\infty}
  L(s) ds$ are finite.  Then $A^{-1} e^{-tA}\Pi_V$ has finite density
  with
  \begin{equation}
    \label{eq:39}
    D(A^{-1} e^{-tA} )  \leq M(t) = \int_t^{+\infty}
    L(s) ds\,.
  \end{equation}
\end{prop}
\begin{proof}
  We follow the proof of Proposition~\ref{prop:2.1}. First by
  \eqref{eq:25} and linearity of trace one has
  \begin{align*}
    \nu_{A^{-1} \Pi_{]\varepsilon,\lambda]}} (\Omega) & =
    \tau(\chi_\Omega A^{-1} \Pi_{]\varepsilon,\lambda]} \chi_\Omega )
    \\
    & = \lambda^{-1} \nu_{\Pi_\lambda}(\Omega) - \varepsilon^{-1}
    \nu_{\Pi_\varepsilon}(\Omega) + \int_{]\varepsilon,\lambda]}
    u^{-2} \nu_{\Pi_\varepsilon}(\Omega) du \\
    & \leq \mu(\Omega) \bigl(\lambda^{-1} F(\lambda) +
    \int_{]\varepsilon,\lambda]} u^{-2} F(u) du \bigr) \\
    & = \mu(\Omega) ( G(\lambda) + F(\varepsilon)/\varepsilon \bigl)
    \longrightarrow \mu(\Omega) G(\lambda) \,,
  \end{align*}
  when $\varepsilon \searrow 0$ since $F(\varepsilon)/\varepsilon \leq
  G(\varepsilon) \rightarrow 0$. Now using an Hilbert basis $f_n$ of
  $\mathcal{H}$, one sees by Beppo-Levi and the spectral theorem that
  \begin{align*}
    \tau(\chi_\Omega A^{-1} \Pi_{]\varepsilon,\lambda]} \chi_\Omega )
    & = \sum_i \|A^{-1/2} \Pi_{]\varepsilon,\lambda]} \chi_\Omega
    f_i\|^2_2 \\
    & \longrightarrow \sum_i \|A^{-1/2} \Pi_\lambda \chi_\Omega
    f_i\|^2_2  \quad \mathrm{when}\quad \varepsilon \searrow 0 \,,\\
    & = \tau(\chi_\Omega A^{-1} \Pi_\lambda \chi_\Omega ) =
    \nu_{A^{-1} \Pi_\lambda}(\Omega) \,.
  \end{align*}
  Therefore we obtain that $\nu_{A^{-1} \Pi_\lambda}(\Omega) \leq
  \mu(\Omega) G(\lambda)$ yielding $D(A^{-1} \Pi_\lambda) \leq
  G(\lambda)$ as claimed.
  
  \smallskip The proof at the heat level also follows
  Proposition~\ref{prop:2.1} and starts from
  \begin{displaymath}
    \nu_{A^{-1} e^{-tA}}(\Omega) = \int_s^{+\infty} \nu_{e^{-sA}}
    (\Omega) ds \,.
  \end{displaymath}
\end{proof}

\begin{rem}
  \label{rem:G=density}
  In these proofs, we note that for invariant operators on groups
  $\Gamma$, endowed with their Haar measure, the previous inequalities
  \eqref{eq:38} and \eqref{eq:39} becomes equalities, as due to
  $\nu_P(\Omega) = D(P) \mu(\Omega)$ in such cases by
  Proposition~\ref{prop:density}.
\end{rem}

The sequel of the proof also follows the pure state case. Let
$\rho^{1/2}$ be the positive square root of $\rho$ and consider for a
measurable $\Omega \subset X $ the following
splitting 
\begin{displaymath}
  \rho^{1/2} \chi_\Omega = \rho^{1/2} A^{1/2} A^{-1/2} \Pi_\lambda \chi_\Omega
  +   \rho^{1/2} \Pi_{>\lambda} \chi_\Omega\,.
\end{displaymath}
Taking Hilbert--Schmidt norm gives
\begin{displaymath}
  \|\rho^{1/2} \chi_\Omega\|_{HS} \leq \|\rho^{1/2} A^{1/2} \|_{2,2}
  \|A^{-1/2} \Pi_\lambda \chi_\Omega \|_{HS}
  +   \|\rho^{1/2} \Pi_{>\lambda} \chi_\Omega\|_{HS} \,.
\end{displaymath}
Then using the following classical properties, see
e.g. \cite[Chap. VI]{Reed-Simon},
\begin{equation}
  \label{eq:40}
  \tau(P^* P) = \|P\|_{HS}^2 = \|P^*\|_{HS}^2 \quad \mathrm{and} \quad
  \|P\|_{2,2}^2 = 
  \|P^*\|_{2,2}^2 = \|P^*P\|_{2,2} \,,
\end{equation}
one finds that for any $\Omega \subset X$
\begin{displaymath}
  \nu_\rho(\Omega) = \|\rho^{1/2} \chi_\Omega\|_{HS}^2 \leq 2
  \|\rho^{1/2} A \rho^{1/2} \|_{2,2} \nu_{A^{-1} \Pi_\lambda} (\Omega) + 2
  \nu_{\Pi_{> \lambda} \rho\Pi_{> \lambda}}( \Omega)\,.
\end{displaymath}
Therefore taking densities gives
\begin{align*}
  D \nu_\rho(x) & \leq 2 \|\rho^{1/2} A \rho^{1/2}\|_{2,2} D
  \nu_{A^{-1}
    \Pi_\lambda}(x) + 2 D \nu_{\Pi_{> \lambda} \rho \Pi_{> \lambda}} (x) \\
  & \leq 2 \|\rho^{1/2} A \rho^{1/2} \|_{2,2} G(\lambda) + 2 D
  \nu_{\Pi_{> \lambda} \rho \Pi_{>\lambda}} (x) \,,
\end{align*}
and finally
\begin{equation}
  \label{eq:41}
  D \nu_\rho(x) \leq 4 D \nu_{\Pi_{> \lambda} \rho \Pi_{> \lambda}}
  (x)  \quad
  \mathrm{if } \quad D \nu_\rho(x) \geq 4  \|\rho^{1/2} A \rho^{1/2}
  \|_{2,2} G(\lambda)   \,.
\end{equation}
By integration on $A = \{ (x,\lambda) \mid D \nu_\rho(x) \geq 4
\|\rho^{1/2} A \rho^{1/2} \|_{2,2} G(\lambda) \}$ this leads to
\begin{align*}
  \int_X G^{-1} \Bigl( \frac{D \nu_\rho(x)}{4 \|\rho^{1/2} A
    \rho^{1/2} \|_{2,2}} \Bigr) d \nu_\rho & = \int_A D \nu_\rho(x) d
  \mu(x)
  d\lambda \quad\mathrm{by\ Fubini},  \\
  & \leq 4 \int_A D \nu_{\Pi_{> \lambda} \rho \Pi_{> \lambda}}
  (x)  d\mu(x) d \lambda \quad \mathrm{by}\quad \eqref{eq:41},\\
  & \leq 4 \int_{X \times \R^+} D \nu_{\Pi_{> \lambda} \rho \Pi_{>
      \lambda}}
  (x)  d\mu(x) d \lambda \\
  & = 4 \int_{\R^+} \tau(\Pi_{> \lambda } \rho \Pi_{> \lambda}) d
  \lambda
  \\
  & = 4 \int_{\R^+} \tau ( \rho^{1/2} \Pi_{> \lambda} \rho^{1/2})
  d\lambda
  \quad \mathrm{by}\quad \eqref{eq:40}\,, \\
  & = 4 \tau \Bigl( \rho^{1/2} \bigl(\int_{\R^+} \Pi_{>\lambda}
  d\lambda \bigr) \rho^{1/2} \Bigr)
  \\
  & = 4 \tau( \rho^{1/2}A \rho^{1/2}) = 4 \mathcal{E}(\rho)\,.
\end{align*}

\subsection{Proof of the mixed state Moser inequalities.}

\label{sec:proof-rho-Moser}
The proofs of Theorem~\ref{thm:F-Sobolev} and
\ref{thm:rho-Moser-partition} follow the same technique. Here one
compares the level sets of $D \nu_\rho(x)$ to $\|\rho\|_{2,2}$ instead
of $\|\rho^{1/2} A \rho^{1/2} \|_{2,2}$ in \eqref{eq:41}. This does
not rely on Proposition~\ref{prop:rho-G} and one can work either with
$\rho= 0$ on $\ker A$ and $\Pi_{]0, \lambda]}$, as before, but also
with a general positive $\rho$ and $\Pi_{[0, \lambda]} $.

In any case, given $\Omega \subset X$, one considers the splitting
\begin{displaymath}
  \rho^{1/2}  \chi_\Omega = \rho^{1/2} \Pi_\lambda \chi_\Omega
  +   \rho^{1/2} \Pi_{>\lambda} \chi_\Omega \,.
\end{displaymath}
Taking Hilbert--Schmidt norms yields
\begin{displaymath}
  \| \rho^{1/2} \chi_\Omega \|_{HS} \leq \| \rho^{1/2} \|_{2,2}
  \|\Pi_\lambda \chi_\Omega \|_{HS} +
  \|\rho^{1/2} \Pi_{>\lambda} \chi_\Omega \|_{HS}  \,,
\end{displaymath}
and using \eqref{eq:40} as above
\begin{equation}
  \label{eq:42}
  \nu_{\rho} (\Omega) = \|\rho^{1/2} \chi_\Omega
  \|_{HS}^2  \leq 2 \|\rho\|_{2,2} \nu_{\Pi_\lambda} (\Omega) + 2 \nu_{\Pi_{>
      \lambda} \rho \Pi_{> \lambda}} (\Omega) \,,
\end{equation}
whence
\begin{align*}
  D \nu_\rho(x) & \leq 2 \|\rho\|_{2,2} D \nu_{\Pi_\lambda}(x)
  + 2 D \nu_{\Pi_{> \lambda} \rho \Pi_{> \lambda}} (x) \\
  & = 2 \|\rho\|_{2,2} F_x(\lambda) + 2 D \nu_{\Pi_{> \lambda} \rho
    \Pi_{>\lambda}} (x) \,.
\end{align*}
Thus in place of \eqref{eq:41} one finds that
\begin{equation}
  \label{eq:43}
  D \nu_\rho(x) \leq 4 D \nu_{\Pi_{> \lambda} \rho \Pi_{> \lambda}}
  (x)  \quad
  \mathrm{if } \quad D \nu_\rho(x) \geq 4  \|\rho
  \|_{2,2} F_x(\lambda)   \,.
\end{equation}
This leads to \eqref{eq:16} by integration on $B = \{ (x,\lambda) \mid
D \nu_\rho(x) \geq 4 \|\rho \|_{2,2} F_x(\lambda) \}$ as done
above. Note that when one works on general $\rho$ and
$\Pi_{[0,\lambda]}$ one has to complete the definition of $F_x^{-1}$
by setting
\begin{equation}
  \label{eq:44}
  F_x^{-1}(u) = 
  \begin{cases}
    0 & \mathrm{if} \quad  u < F_x(0) = D \nu_{\Pi_{\ker A}} (x)\,,\\
    \sup \{\lambda \mid F_x(\lambda) \leq u \} & \mathrm{elsewhere}\,.
  \end{cases}
\end{equation}
This means that the inequality \eqref{eq:16} cuts off small values of
$D \nu_\rho$ in this setting.

\medskip

$\bullet$ For the discrete Moser inequality, \eqref{eq:42} provides
\begin{equation}
  \label{eq:45}
  \nu_\rho(\Omega) \leq 4 \nu_{\Pi_{>\lambda} \rho \Pi_{> \lambda}} (\Omega) \quad
  \mathrm{if } \quad \nu_\rho(\Omega) \geq 4 \| \rho\|_{2,2}
  F_\Omega(\lambda) \,,
\end{equation}
in place of the local \eqref{eq:43}. Given a partition $X = \bigsqcup
\Omega_i$, and summing \eqref{eq:45} on
\begin{displaymath}
  C= \{ (i, \lambda )\mid
  \nu_\rho(\Omega_i) \geq 4 \| \rho\|_{2,2} F_{\Omega_i}(\lambda) \}
  \subset I \times \R^+
\end{displaymath}
leads now to \eqref{eq:17}, where again one defines in general
\begin{equation}
  \label{eq:46}
  F^{-1}_\Omega (u) = 
  \begin{cases}
    0 & \mathrm{if} \quad  u < F_\Omega(0) = \nu_{\ker A} (\Omega) \,,\\
    \sup \{\lambda \mid F_\Omega(\lambda) \leq u \} &
    \mathrm{elsewhere}\,.
  \end{cases}
\end{equation}

\smallskip

$\bullet$ When the density $\rho$ is supported in $\Omega$, one has
$\nu_\rho(\Omega) = \nu_\rho(X) = \tau (\rho)$ and the Moser
inequality \eqref{eq:17} writes
\begin{equation}
  \label{eq:47}
  F^{-1}_{\Omega} \Bigl( \frac{\tau(\rho)}{4 \|\rho\|_{2,2}} \Bigr)
  \leq 4 \frac{\mathcal{E}(\rho)}{\tau (\rho)} = 4 \langle A\rangle_\rho .
\end{equation}
This yields the Faber--Krahn inequality \eqref{eq:18} when $\langle
A\rangle_\rho $ is finite since by right continuity of $F_\Omega$ and
\eqref{eq:46}, one has $F_\Omega (F^{-1}_\Omega(\lambda) )\geq
\lambda$ when $F_\Omega^{-1}(\lambda)$ is finite.

\subsection{Note on constants}
\label{sec:note-constants}

We observe that one can balance differently the multiplicative
constants in the proofs of inequalities. Given $\varepsilon \in
]0,1[$, one can replace for instance \eqref{eq:45} by
\begin{displaymath}
  \varepsilon^2  \nu_\rho(\Omega) \leq \nu_{\Pi_{>\lambda} \rho \Pi_{>
      \lambda}} (\Omega) \quad 
  \mathrm{if }\quad  (1- \varepsilon)^2 \nu_\rho(\Omega) \geq  \|
  \rho\|_{2,2} 
  F_\Omega(\lambda) \,,
\end{displaymath}
yielding
\begin{displaymath}
  \sum_i F_{\Omega_i}^{-1} \Bigl( \frac{(1- \varepsilon)^2 \nu_\rho
    (\Omega_i)}{ 
    \|\rho\|_{2,2}} \Bigr) \nu_\rho (\Omega_i) \leq \varepsilon^{-2}
  \mathcal{E}(\rho) 
\end{displaymath}
instead of \eqref{eq:17}, and in particular to the Faber--Krahn
inequality
\begin{equation}
  \label{eq:48}
  (1-\varepsilon)^2 F_\Omega^{\dim} (\lambda) \leq F_\Omega
  (\varepsilon^{-2} \lambda) \leq \mu(\Omega) F
  (\varepsilon^{-2} \lambda)
\end{equation}
in place of \eqref{eq:19}. With respect to the ideal bound
\eqref{eq:20}, one sees in \eqref{eq:48} that tightening the energy
multiplicative constant to $1$ blows up the volume one, and reversely.

\section{Relationships between inequalities}
\label{sec:relat-betw-ineq}

\subsection{Spectral versus heat decay}
\label{sec:spectral-heat-decay}

We now compare the various results obtained.

We first consider the two Theorems~\ref{thm:H-Sobolev}
and~\ref{thm:N-Sobolev}. They both state Sobolev inequalities for
functions starting either from the heat or spectral decay. One can
compare $F$ and $G$ to $L$ and $M$ through Laplace transform of
associated measures.
\begin{prop}
  \label{prop:3:2}
  $\bullet $ In any case it holds that
  \begin{gather}
    \label{eq:49}
    L(t) \leq \mathcal{L}(dF)(t) = \int_0^{+\infty} e^{-\lambda
      t} d F(\lambda) \\
    \label{eq:50}
    M(t) \leq \mathcal{L}(dG)(t) = \int_0^{+\infty} e^{-\lambda t} d
    G(\lambda)\,.
  \end{gather}
  
  $\bullet$ If $A$ is an invariant operator acting on $\mathcal{H} =
  L^2(\Gamma, H)$ over a locally compact group $\Gamma$, then reverse
  inequalities hold up to the multiplicative factor $n = \dim H$, i.e.
  \begin{displaymath}
    \mathcal{L}(dF) \leq n L \quad \mathrm{and} \quad \mathcal{L}(dG)
    \leq n M\,.
  \end{displaymath}
  Moreover $G(y) \leq n e M(y^{-1})$ and H-Sobolev inequality
  \eqref{eq:3} implies N-Sobolev \eqref{eq:9}, up to multiplicative
  constants.

  $\bullet$ Reversely, for any operator, if $G$ satisfies the
  exponential growing condition :
  \begin{equation}
    \label{eq:51}
    \exists C \ \mathrm{such\ that\
    }\forall u, y > 0\,,\  G(uy) \leq
    e^{C u} G(y) \,,
  \end{equation}
  then $ M(y^{-1}) \leq 3 G(2Cy)$. Hence $H$ and $N$-Sobolev are
  equivalent on groups in that case.
\end{prop}
\begin{proof}
  $\bullet $ By spectral calculus $e^{-tA}\Pi_V = \int_0^{+\infty}
  e^{-t\lambda} d \Pi_\lambda = t \int_0^{+\infty} e^{-t \lambda}
  \Pi_\lambda d \lambda$, hence
  \begin{displaymath}
    L(t) = \| e^{-t A} \Pi_V\|_{1, \infty} \leq t \int_0^{+\infty} e^{-t
      \lambda} \|\Pi_\lambda\|_{1, \infty} d\lambda =
    \mathcal{L}(dF)(t) \,,
  \end{displaymath}
  and thus
  \begin{displaymath}
    M(t) = \int_t^{+\infty} L(s) ds  \leq
    \int_t^{+\infty} \int_0^{+\infty} e^{-\lambda s } d F(\lambda)
    ds = \int_0^{+\infty} \frac{e^{-\lambda t}}{\lambda} dF(\lambda) =
    \mathcal{L}(dG) (t)\,.  
  \end{displaymath}
  
  $\bullet$ For positive invariant operators $P$ on groups, the
  ultracontractive norm $\|P\|_{1, \infty}$ is pinched between the
  density $D (P)$ and $n D(P)$ by \eqref{eq:7}. This gives the reverse
  inequalities by positive linearity of $D(P)$ on such operators. In
  particular one gets
  \begin{align*}
    n M( y^{-1}) & \geq \int_0^{+\infty} e^{-\lambda/y} dG (\lambda) =
    y^{-1}
    \int_0^{+\infty} e^{-\lambda/y} G(\lambda) d \lambda \\
    & \geq y^{-1} \int_y^{+\infty} e^{-\lambda/y} G(y) d\lambda =
    e^{-1} G(y).
  \end{align*}
  Therefore $N(y) = y / M^{-1}(y^{-1}) \leq y G^{-1}(ey) = e^{-1}
  H(ey)$ and $H$-Sobolev implies
  \begin{displaymath}
    \int_X N \Bigl( \frac{|f(x)|^2}{4e \mathcal{E}(f)} \Bigr) d \mu
    \leq e^{-1}\,. 
  \end{displaymath}
  
  $\bullet$
  If $G$ satisfies the growing condition \eqref{eq:51}, one has by
  \eqref{eq:50}
  \begin{align*}
    M(1/y) & \leq \int_0^{+\infty} e^{-\lambda/y} d G(\lambda) =
    \int_0^{+\infty} e^{-u} G(u y) d u \\
    & \leq \int_0^{2C} e^{-u} G(2C y) du + \int_{2C}^{+\infty} e^{-
      u/2}
    G(2C y) du  \\
    & \leq 3 G(2Cy) \,.
  \end{align*}
\end{proof}

We note that it may happen that $N \ll H$ for very thin near--zero
spectrum. In an extreme case there may be a gap in the spectrum,
i.e. $A \geq \lambda_0 > 0$, hence $F = G =0$ on $[0,\lambda_0[$ and
$H(y) \geq \lambda_0 y$, while $L(t) \asymp C e^{-c t} $, $M (t)
\asymp C' e^{-c t}$ and $N(y) \asymp C'' y / \ln (y/C')$.

\smallskip

We mention also that a similar statement holds at the mixed state
level between the two $\rho$-Sobolev inequalities obtained from the
spectral density $F(\lambda)= D(\Pi_\lambda)$ or the heat density
$L(t) = D(e^{-t A})$. Namely the $F$ version is stronger in general to
the $L$ one on groups, although of the same strength under the same
growing assertion \eqref{eq:51} on $G$, that excludes extremely thin
spectrum near zero.

\subsection{From Sobolev to Faber--Krahn}
\label{sec:sobolev-faber-Krahn}

We have seen in the introduction how the mixed state Sobolev
inequality implies a global Faber--Krahn inequality \eqref{eq:13} in
the case of a polynomial spectral density. We extend this to other
profiles and compare the result to Corollary~\ref{cor:Faber-Krahn},
obtained from the mixed state Moser approach.

\begin{prop}
  \label{prop:Sobolev-Faber-Krahn}
  With the assumptions of Theorem~\ref{thm:rho-Sobolev}, suppose
  moreover that $\rho$ is supported in $\Omega$ of finite volume. Then
  \begin{equation}
    \label{eq:52}
    \tau(\rho) \leq 8 \mu(\Omega) \|\rho^{1/2} A \rho^{1/2 }\|_{2,2} G
    (8 \langle A \rangle_\rho)
  \end{equation}
  $\bullet$ In particular the following global Faber--Krahn inequality
  holds
  \begin{equation}
    \label{eq:53}
    F_\Omega^{\dim} (\lambda) \leq 8 \mu(\Omega) \lambda G(8\lambda) \,,
  \end{equation}
  which is a priori weaker than \eqref{eq:19}, since $F(\lambda) \leq
  \lambda G(\lambda)$ in general.
\end{prop}

\begin{proof}
  When the function $t G^{-1}(t) $ is convex, Jensen inequality easily
  gives the result, with better coefficients. Without convexity
  assumption, we can argue again as in
  \S\ref{sec:related-inequalities}. Observe that for all non-negative
  $s$ and $t$ one has
  \begin{displaymath}
    st  \leq s G(s) + t G^{-1}(t)\,.
  \end{displaymath}
  Applying to $t= \frac{D \nu_\rho(x)}{4 \|\rho^{1/2} A \rho^{1/2
    }\|_{2,2}}$ and integrating over $\Omega$ yields
  \begin{align*}
    \frac{s \tau(\rho)}{4 \|\rho^{1/2} A \rho^{1/2 }\|_{2,2}} - s G(s)
    \mu(\Omega) & \leq \int_X \frac{D \nu_\rho(x)}{4 \|\rho^{1/2} A
      \rho^{1/2 }\|_{2,2}} G^{-1}\Bigl( \frac{D \nu_\rho(x)}{4
      \|\rho^{1/2} A \rho^{1/2 }\|_{2,2}}
    \Bigr) d \mu \\
    & \leq \frac{\tau(A\rho)}{ \|\rho^{1/2} A \rho^{1/2 }\|_{2,2}}
    \quad \mathrm{by} \quad \eqref{eq:9} \,.
  \end{align*}
  Using
  \begin{displaymath}
    s \nearrow G^{-1}\Bigl(\frac{\tau(\rho)}{8 \mu(\Omega) \|\rho^{1/2}
      A \rho^{1/2 }\|_{2,2}} \Bigr) = 
    \sup \Bigl\{s \mid G(s) \leq \frac{\tau(\rho)}{8 \mu(\Omega) \|\rho^{1/2}
      A \rho^{1/2 }\|_{2,2}} \Bigr\} \,,
  \end{displaymath}
  gives
  \begin{displaymath}
    G^{-1}\Bigl(\frac{\tau(\rho)}{8 \mu(\Omega) \|\rho^{1/2}
      A \rho^{1/2 }\|_{2,2}} \Bigr) \leq 8 \langle A \rangle_\rho \,,
  \end{displaymath}
  and \eqref{eq:52} since $G(G^{-1}(s)) \geq s$ when $G^{-1}(s)$ is
  finite by right continuity of $G$.
\end{proof}

Observe that one may have $F (\lambda) \ll \lambda G(\lambda)$ for
very thick near-zero spectrum, even when $G$ converges. For instance
if $F(\lambda) = \lambda / \ln^2 \lambda$ then $\lambda G(\lambda)=
(-\ln \lambda + 1 )F(\lambda)$. Except this ``low dimensional''
phenomenon, one has $\lambda G(\lambda) \underset{0}{\asymp}
F(\lambda)$ in the other cases, and thus the two global Faber--Krahn
inequalities \eqref{eq:19} and \eqref{eq:53} obtained through
$\rho$-Sobolev or Moser inequalities have same strength. For instance
this holds if $F(\lambda) \underset{0}{\sim} \lambda^{1+\varepsilon}
\varphi(\lambda)$ for some $\varepsilon > 0$ and an increasing
$\varphi>0$. This comes from the following remark.
\begin{prop}
  \label{prop:3:1}Suppose there exists $\varepsilon > 0$ such that,
  for small $\lambda$, $F$ satisfies the growing condition $
  F(2\lambda) \geq 2 (1+ \varepsilon) F(\lambda)$, then
  $(2+\varepsilon^{-1}) F(\lambda) \geq \lambda G(\lambda) \geq
  F(\lambda)$.
\end{prop}
\begin{proof}
  By \eqref{eq:26}, one has
  \begin{align*}
    G(\lambda) & = \int_0^\lambda \frac{dF(u)}{u} =
    \frac{F(\lambda)}{\lambda} + \int_0^\lambda \frac{F(u)}{u^2} du \\
    & = \frac{F(\lambda)}{\lambda} + \Bigl( \int_0^{\lambda/2} +
    \int_{\lambda/2}^\lambda \Bigr) \frac{F(u)}{u^2} du \\
    & \leq \frac{2 F(\lambda)}{\lambda} + \int_0^{\lambda/2}
    \frac{F(2u)}{2 (1+\varepsilon) u^2} du \quad \mathrm{ by\
      hypothesis\ on\ }F\,,\\
    & \leq \frac{2 F(\lambda)}{\lambda} + \frac{1}{1+\varepsilon}
    \Bigl( G(\lambda) - \frac{ F(\lambda)}{\lambda} \Bigr)\,,
  \end{align*}
  leading to $\lambda G(\lambda) \leq (2 + \varepsilon^{-1})
  F(\lambda)$\,.
\end{proof}
As a curiosity, we note that under the growing hypothesis on $F$
above, the spectral density of states $F$ and the spatial repartition
function $H$ have symmetric expressions with respect to $G$ and
$G^{-1}$. Indeed, one has simply there
\begin{equation}
  \label{eq:54}
  F(\lambda) \asymp \lambda G(\lambda) \quad \mathrm{while } \quad
  H(x) = x G^{-1}(x)\,.
\end{equation}

  

\subsection{Discrete and integral $\rho$-Moser inequalities}
\label{sec:discr-inte-rho}

We now study the relationships between the two versions of the
$\rho$-Moser inequalities given in Theorem~\ref{thm:F-Sobolev}
and~\ref{thm:rho-Moser-partition}.

A first remark is that the discrete case may be seen as an instance of
the integral version.  Indeed if $X = \bigsqcup_I \Omega_i$, then one
may split $\mathcal{H} = L^2(X, \mu) \otimes H = \bigoplus_I
L^2(\Omega_i, \mu)\otimes H$ and, given a positive $P$ on
$\mathcal{H}$, defines a measure $\nu'_P$ on $I$ by
\begin{displaymath}
  \nu'_P(J) = \sum_{i \in J} \tau( \chi_{\Omega_i} P \chi_{\Omega_i})
  =  \nu_P(\cup_J \Omega_i) \,. 
\end{displaymath}
Then the density function of $\nu'_P$ writes
\begin{displaymath}
  D\nu'_P(i) = \tau(\chi_{\Omega_i} P
  \chi_{\Omega_i} )= \nu_P(\Omega_i) \,.
\end{displaymath}
In this setting the integral formula \eqref{eq:16} yields the discrete
one \eqref{eq:17}.

\smallskip

Another feature of these formulae is that, up to multiplicative
constants, the integral formula dominates all the discrete one,
whatever the partition of $X$.

\begin{prop}
  \label{prop:continuous-discrete-Moser}
  With notations of Theorem~\ref{thm:F-Sobolev}, given a measurable
  $\Omega$ in $X$, one has
  \begin{equation}
    \label{eq:55}
    F_\Omega^{-1}\Bigl( \frac{\nu_\rho
      (\Omega)}{8\|\rho\|_{2,2}} \Bigr) \nu_\rho (\Omega)  \leq
    2 \int_\Omega F_x^{-1} \Bigl( \frac{D \nu_\rho(x)}{4 \|\rho\|_{2,2}} \Bigr)
    d \nu_\rho \,.
  \end{equation}
  In particular, for any partition $X = \bigsqcup_I \Omega_i$ one has
  \begin{displaymath}
    \sum_i F_{\Omega_i}^{-1}\Bigl( \frac{\nu_\rho
      (\Omega_i)}{8\|\rho\|_{2,2}} \Bigr) \nu_\rho (\Omega_i)  \leq
    2 \int_X F_x^{-1} \Bigl( \frac{D \nu_\rho(x)}{4 \|\rho\|_{2,2}} \Bigr)
    d \nu_\rho  \,.
  \end{displaymath}
\end{prop}
Hence Theorem~\ref{thm:F-Sobolev} implies
Theorem~\ref{thm:rho-Moser-partition} and
Corollary~\ref{cor:Faber-Krahn}, up to weaker constants $8$ instead of
$4$.

\begin{proof}
  Given $x \in X$, one has for all non negative $s$ and $t$
  \begin{displaymath}
    st \leq s F_x(s) + t F_x^{-1}(t) \,,
  \end{displaymath}
  where $F_x (\lambda) = D \nu_{\Pi_\lambda}(x)$. Applying to $t =
  \frac{D \nu_\rho(x)}{4 \|\rho\|_{2,2}}$ and integrating on $\Omega$
  yields
  \begin{displaymath}
    s \frac{\nu_\rho(\Omega)}{4 \|\rho\|_{2,2}} - s
    F_\Omega(s) \leq  \int_\Omega F_x^{-1} \Bigl( \frac{D
      \nu_\rho(x)}{4 \|\rho\|_{2,2}} \Bigr) 
    \frac{d \nu_\rho}{4 \|\rho\|_{2,2}} \,.   
  \end{displaymath}
  This gives \eqref{eq:55} using
  \begin{displaymath}
    s \nearrow F_\Omega^{-1} \Bigl( \frac{\nu_\rho(\Omega)}{8
      \|\rho\|_{2,2}} \Bigr) = \sup \Bigl\{ s \mid F_\Omega( s) \leq 
    \frac{\nu_\rho(\Omega)}{8 
      \|\rho\|_{2,2}} \Bigr\} \,.
  \end{displaymath}
\end{proof}

\subsection{Ultracontractivity, density and von Neumann trace}
\label{sec:ultrac-dens-von}

We now discuss Proposition~\ref{prop:density} that relates the three
measurements of positive operators used here: through
ultracontractivity, density function, or $\Gamma$-trace.

At first, \eqref{eq:6} comes directly from the definition \eqref{eq:5}
in the form $\nu_P(\Omega) = \tau(P^{1/2} \chi_\Omega P^{1/2})$. This
expression of the density is useful when $P$ is a positive compact
operator. Namely if $P = \sum_i \lambda_i \Pi_{e_i}$ is the spectral
decomposition of $P$, it reads
\begin{displaymath}
  D \nu_P(x) = \sum_{i} \lambda_i \|e_i(x)\|_H^2 \quad\mathrm{almost \
    everywhere.}  
\end{displaymath}
Also, \eqref{eq:6} clearly implies that $\|P\|_{1, \infty} =
\|P^{1/2}\|^2_{2,\infty} \leq D(P)= \mathop{\mathrm{supess}}D
\nu_P(x)$ holds in general. For the opposite inequality, we suppose $H
$ is finite $d$-dimensional. Then given a basis $(e_i)$ of
$\mathcal{H}$ and $(h_j)$ of $H$, one has by \eqref{eq:6}
\begin{displaymath}
  D\nu_P(x)  = \sum_i\|P^{1/2} e_i(x)\|_H^2 = \sum_{j=1}^d \sum_i
  \langle( P^{1/2} e_i)(x), h_j \rangle^2 
\end{displaymath}
where, for each $j$, by Cauchy--Schwarz
\begin{align*}
  \sum_i \langle (P^{1/2} e_i)(x), h_j \rangle^2 & = \sup_{\sum c_i^2
    \leq 1} \bigl( \sum_i c_i \langle (P^{1/2} e_i)(x),
  h_j \rangle \bigr)^2  \\
  & = \sup_{\sum c_i^2 \leq 1} \Bigl( \langle \bigl(P^{1/2} (\sum_i
  c_ie_i)\bigr) (x),
  h_j \rangle \Bigr)^2\\
  & = \sup_{\|f\|_2 \leq 1} \bigl( \langle (P^{1/2} f)(x), h_j\rangle
  \bigr)^2 \\
  & \leq \|P^{1/2}\|_{2,\infty}^2 = \|P\|_{1,\infty}\,.
\end{align*}
This gives $D\nu_P(x) \leq d \|P\|_{1, \infty}$ as needed.

\smallskip

We would like to illustrate here the relevance of the spectral density
$D(\Pi_\lambda)$ when dealing with general mixed state inequalities,
while the ultracontractive norm $\|\Pi_\lambda\|_{1,\infty}$ is
adapted to functions. Indeed, suppose that $A$ is a scalar positive
operator with finite $F(\lambda)= \|\Pi_\lambda\|_{1,\infty} =
D(\Pi_\lambda)$ and consider $n$-copies $A_n$ of $A$ acting diagonally
on $\mathcal{H}_n = L^2(X , \C^n)$. Then one finds easily that
\begin{displaymath}
  \|\Pi_\lambda(A_n) \|_{1,\infty} = \|\Pi_\lambda(A) \|_{1,\infty} (=
  \|\Pi_\lambda(A)\|_{2,\infty}^2) 
  \quad \mathrm{while} \quad D(\Pi_\lambda(A_n)) = n D(\Pi_\lambda(A)) \,,
\end{displaymath}
if one sets $\|f\|_\infty = \sup_X \|f(x)\|_2$ and $\|f\|_1 = \int_X
\|f(x)\|_2 d\mu$ with the hermitian norm of $\C^n$. Hence the Sobolev
inequality \eqref{eq:3} is independent of the phase space dimension
$n$, as the Faber--Krahn inequality \eqref{eq:34} for the \emph{first}
eigenvalue $\lambda_{1,n}$ of $A_n$ on a domain $\Omega$, that writes
\begin{displaymath}
  4 \mu(\Omega) F (8 \lambda_{1,n}(\Omega)) \geq 1 \,.
\end{displaymath}
In comparison, the global Faber--Krahn inequality \eqref{eq:19} only
yields
\begin{displaymath}
  4n \mu(\Omega) F(4\lambda_{1,n}(\Omega)) \geq 1
\end{displaymath}
for the first eigenvalue, but also gives the linear bound in $n$ of
the whole spectral distribution, namely: $F_{\Omega,
  A_n}^{\dim}(\lambda) \leq 4n \mu(\Omega) F(4 \lambda)$.


\smallskip

We come back to the study of the density of invariant operators on
locally compact groups $\Gamma$ with their Haar measure. In such a
case, the measure $\nu_P(\Omega) = \tau(\chi_\Omega P \chi_\Omega) $
is clearly invariant too, thus its density function is constant, as
claimed in Proposition~\ref{def:P-measure}. Moreover, when $\Gamma$ is
discrete, this density coincides with von Neumann trace since
\begin{displaymath}
  D(P)= D \nu_P(e) = \tau_\mathcal{H}(\chi_e P \chi_e ) = \tau_H
  (K_P(e,e)) \overset{\mathrm{def}}{=} \tr_\Gamma(P) \,,
\end{displaymath} 
where $K_P$ is the kernel of $P$.  More generally, one can
characterise finite density operators on non necessarily discrete
groups as follows.
\begin{prop}
  \label{prop:traceclass}
  Let $\Gamma$ be a locally compact group with its Haar measure and
  $Q$ be a bounded $\Gamma$-invariant operator on $\mathcal{H} =
  L^2(\Gamma, \mu) \otimes H$. Let $E$ be the space of
  Hilbert--Schmidt operators on $H$ endowed with the Hilbert--Schmidt
  norm.

  $\bullet$ Then $P= Q^*Q$ has a finite density $D(P)$ iff $Q$ has a
  kernel $K_Q(x,y) = k_Q(y^{-1}x)$ with $k_Q \in L^2(\Gamma, E)$ and $
  D(P) = \int_\Gamma \|k_Q(x)\|_{HS}^2 d \mu $.
  
  $\bullet$ If moreover $\Gamma$ is \emph{unimodular},
  one has $D (Q^*Q) = D(Q Q^*)$ and the density actually defines a
  faithful trace in that case.
\end{prop}

\begin{rem}
  We recall that this last trace property allows to get a meaningful
  notion of dimension for closed $\Gamma$-invariant subspaces $L
  \subset \mathcal{H} = L^2(\Gamma) \otimes H$. Indeed, one sets then
  $\dim_\Gamma L = D(\Pi_L) $. This satisfies the key property
  $\dim_\Gamma f(L) = \dim_\Gamma L$ for any closed densely defined
  invariant injective operator $f: L \rightarrow H$, see e.g. \cite[\S
  2]{Pansu1} or \cite[\S 3.2, \S 6]{Rumin05}.
\end{rem}

\begin{proof}
  We recall that an operator is Hilbert--Schmidt if and only if it
  possesses an $L^2$ kernel, see e.g. \cite[Chap VI]{Reed-Simon}. Then
  by definition
  \begin{align*}
    \nu_P(\Omega) & = \tau_\mathcal{H}(\chi_\Omega Q^*Q
    \chi_\Omega) = \|Q \chi_\Omega\|_{\mathcal{H}S}^2  \\
    & = \int_{\Gamma \times \Omega} \|K_Q(x,y)\|_{HS}^2 d \mu (x)
    d\mu(y) \\
    & = \mu(\Omega) \int_\Gamma \|k_Q(x)\|_{HS}^2 d \mu(x) \quad
    \mathrm{by\ invariance}\,.
  \end{align*}
  Hence $D(P) = \int_\Gamma \|k_Q(x)\|_{HS}^2 d \mu(x)$.
  Moreover, one has
  $$\|k_Q(x)\|_{HS} = \|(k_Q(x))^*\|_{HS} = \|k_{Q^*}(x^{-1})\|_{HS}$$
  giving
  \begin{displaymath}
    D(Q^* Q) = \int_\Gamma\|k_Q(x)\|_{HS}^2 d \mu(x) = \int_\Gamma
    \|k_{Q^*}(y)\|_{HS}^2 d \mu(y) = D(Q Q^*)
  \end{displaymath}
  on unimodular groups since there $\mu(\Omega^{-1}) = \mu(\Omega)$.
\end{proof}

One can finally express the density using Fourier analysis on some
family of groups. Namely, following Dixmier \cite[\S 18.8]{Dixmier},
if the group $\Gamma$ is \emph{locally compact, unimodular and
  post\-li\-mi\-nai\-re}, there exists a Plancherel measure $\mu^*$ on
its unitary dual $\widehat \Gamma$, together with a Plancherel formula
on $L^2(\Gamma)$. In particular on positive operators $P= Q^*Q$, one
has using Proposition~\ref{prop:traceclass}
\begin{equation}
  \label{eq:56}
  D(P) = \|k_Q\|_2^2 = \int_{\widehat \Gamma} 
  \|\widehat{k_Q}(\xi)\|_{HS}^2 d \mu^*(\xi) \,.
\end{equation}
This allows to estimate the spectral density $F(\lambda)$ in some
simple cases as in $\R^n$.

\subsection{Illustration and comparison with known inequalities on
  $\R^n$}
\label{sec:illustr-comp-rn}



For instance, in the case of the Laplacian $\Delta$ on $\R^n$, the
spectral space $E_\lambda(\Delta)$ is the Fourier transform of
functions supported in the ball $B(0, \sqrt\lambda)$ in
$(\widehat{\R^n}, d\mu^*) \simeq (\R^n, (2\pi)^{-n} dx)$, hence
$\widehat{k_{\Pi_\lambda}} = \chi_{B(0, \sqrt\lambda)}$ and
\eqref{eq:56} provides
\begin{equation}
  \label{eq:57}
  F(\lambda) = \mu^*(B(0, \sqrt\lambda)) = C_n \lambda^{n/2} ,
\end{equation}
with $C_n= (2\pi)^{-n} \mathrm{vol}(B_n)$. This leads to
\begin{displaymath}
  G(\lambda) =
  \frac{n C_n}{n-2} \lambda^{n/2 -1}\quad \mathrm{and}\quad H(x) = x
  G^{-1}(x) = \Bigl(\frac{n-2}{n C_n}\Bigr)^{\frac{2}{n-2}} x^{\frac{n}{n-2}} \,,
\end{displaymath}
so that finally \eqref{eq:3} gives the classical Sobolev inequality in
$\R^n$
\begin{displaymath}
  \|f\|_{2n/(n-2)} \leq \frac{1}{\pi} \Bigl(
  \frac{n\,\mathrm{vol}(B_n)}{n-2} \Bigr)^{\frac{1}{n}} \|df\|_2  =
  D_n \|df\|_2\,. 
\end{displaymath}
One finds that the constant $D_n$ has the correct rate of decay in
$n$, namely $D_n \sim_{+\infty} \sqrt{\frac{2 e}{n \pi}}$.  While
according to \cite{Aubin}, the best constant here is $D_n^* =
2(n(n-2))^{-1/2} \mathrm{area}(S_n)^{-1/n}$, and satisfies $D_n^*
\sim_{+\infty} \sqrt{\frac{2}{n \pi e}}$.

We now consider the Moser inequality \eqref{eq:31}. On functions this
gives the classical $L^2$-Moser inequality with constants with the
right decay in $n$ on $\R^n$. Indeed from \eqref{eq:57}, one finds
that
\begin{displaymath}
  \|f \|_{2+4/n}^{2+4/n} \leq 4^{1+2/n} C_n^{2/n}\|f\|_2^{4/n}
  \|df\|_2^2 = E_n \|f\|_2^{4/n}
  \|df\|_2^2 \,,
\end{displaymath}
with $E_n \sim_{+\infty} \frac{2e}{n \pi}$ while, following Beckner,
see \cite{Beckner} or \cite[Appendix]{Coulhon-Grigorian-Levin}, the
best constants in the $L^2$-Moser inequality are asymptotic to
$\frac{2}{n \pi e}$.

\medskip

Still on $\R^n$, one can get some general algebraic expression of
$F(\lambda)$ for positive invariant differential operator $A = \sum_I
a_I \partial_{x_I}$. Let $\sigma(A)(\xi) = \sum_I a_I (i\xi)^I$ be its
polynomial symbol. Then again the spectral space $E_\lambda(A)$
consists in functions whose Fourier transform is supported in
\begin{displaymath}
  D_\lambda= \{\xi \in \R^n \mid \sigma(A)(\xi) \leq \lambda\}
\end{displaymath}
and as above
\begin{displaymath}
  F(\lambda) = (2\pi)^{-n} \mathrm{vol}(D_\lambda). 
\end{displaymath}
The asymptotic behaviour of $F(\lambda)$ when $\lambda \searrow 0$ can
be obtained from the resolution of the singularity of the polynomial
$\sigma(A)$ at $0$. Indeed, there exists $\alpha \in \mathbb{Q}^+$ and
$k \in [0,n-1]\cap \N$ such that
\begin{displaymath}
  F(\lambda) \underset{\lambda \rightarrow 0^+}{\sim} C \lambda^\alpha |\ln
  \lambda|^k  \,,
\end{displaymath}
see e.g. Theorem 7 in \cite[\S21.6]{Arnold}. Moreover, under a
non-degeneracy hypothesis on $\sigma(A)$, the exponents $\alpha$ and
$k$ can be read from its Newton polyhedra. Then if $\alpha > 1$,
Proposition~\ref{prop:3:1} yields that $G(\lambda)
\underset{0}{\asymp} \lambda^{\alpha-1} |\ln \lambda|^k$. Therefore
$G^{-1}(u) \underset{0}{\asymp} u^{1/(\alpha-1)} |\ln
u|^{-k/(\alpha-1)}$ and finally the $H$-Sobolev inequality
\eqref{eq:3} is governed in small energy by the function
\begin{displaymath}
  H(u) \asymp u^{\frac{\alpha}{\alpha-1}}
  |\ln(u)|^{-\frac{k}{\alpha-1}} \quad \mathrm{for }\quad u \ll 1\,. 
\end{displaymath}

\section{Spectral density and cohomology}
\label{sec:spectr-dens-cohom}

We conclude with an application of the Sobolev inequalities in the
pure state setting.  Let $K$ be a finite simplicial complex and
consider a covering $\Gamma \rightarrow X \rightarrow K$. Let $d_k$ be
the coboundary operator on $k$-cochains $X^k$ of $X$. As a purely
combinatorial and local operator, it acts boundedly on all
$\ell^p$-spaces of cochains $\ell^p X^k$, see
e.g. \cite{Bourdon-Pajot,Pansu1}.

Let $F_{\Gamma,k}(\lambda)$ denotes the $\Gamma$-trace of the spectral
projector $\Pi_\lambda = \chi_A(]0, \lambda])$ of $A= d_k^* d_k$. By
Proposition~\ref{prop:density} this function coincides with the
density of $\Pi_\lambda$ relatively to $\Gamma$ and is also
equivalent, up to multiplicative constants, to the ultracontractive
spectral decay $F(\lambda) = \|\Pi_\lambda\|_{1, \infty}$. Thus
Theorem~\ref{thm:1.4} is a direct application of
Theorem~\ref{thm:H-Sobolev} or~\ref{thm:rho-Sobolev} in the polynomial
case. This statement compares two measurements of the torsion of
$\ell^2$-cohomology $ T_2^{k+1} = \overline{ d_k(\ell^2)}^{\ell^2} /
d_k(\ell^2) $ that share some geometric invariance. We describe this
more precisely.

\medskip

We first recall the main invariance property of
$F_{\Gamma,k}(\lambda)$. We say that two increasing functions $f, g:
\R^+ \rightarrow \R^+$ are equivalent if there exists $C \geq 1$ such
that $f(\lambda/C) \leq g(\lambda) \leq f(C \lambda)$ for $\lambda$
small enough. According to
\cite{Efremov,Gromov-Shubin,Gromov-Shubin-erratum} we have:
\begin{thm}
  Let $K$ be a finite simplicial complex and $ \Gamma \rightarrow X
  \rightarrow K $ a covering. Then the equivalence class of
  $F_{\Gamma,k}$ only depends on $\Gamma$ and the homotopy class of
  the $(k+1)$-skeleton of $K$.
\end{thm}
One tool in the proof is the observation that an homotopy of finite
simplicial complexes $X$ and $Y$ induces bounded $\Gamma$-invariant
homotopies between the Hilbert complexes $( \ell^2 X^k, d_k)$ and
$(\ell^2 Y^k, d'_k)$. That means there exist $\Gamma$-invariant
bounded maps
\begin{displaymath}
  f_k : \ell^2 X^k \rightarrow \ell^2 Y^k \quad
  \mathrm{and} \quad g_k : \ell^2  Y^k
  \rightarrow \ell^2  X^k
\end{displaymath}
such that
\begin{displaymath}
  f_{k+1} d_k = d'_kf_k \quad \mathrm{and}\quad g_{k+1} d'_k =
  d_k g_k
\end{displaymath}
and
\begin{displaymath}
  g_k f_k = \mathrm{Id} + d_{k-1}h_k + h_{k+1} d_k \quad \mathrm{and}
  \quad
  f_k g_k = \mathrm{Id} + d'_{k-1} h'_k + h'_{k+1} d'_k
\end{displaymath} 
for some bounded maps
\begin{displaymath}
  h_k : \ell^2  X^k \rightarrow \ell^2
  X^{k-1} \quad \mathrm{and} \quad h'_k : \ell^2 
  Y^k \rightarrow 
  \ell^2  Y^{k-1}. 
\end{displaymath}
Actually all these maps are purely combinatorial and local, see
e.g. \cite[\S1]{Bourdon-Pajot}, and thus extend on all $\ell^p$ spaces of
cochains.

One can show a similar invariance property of the inclusion
\eqref{eq:21} we recall below, but that holds more generally on
uniformly \emph{locally finite simplicial complexes}, without
requiring a group invariance. These are simplicial complexes such that
each point lies in a bounded number $N(k)$ of $k$-simplexes.

\begin{prop}
  \label{prop:invariance-torsion}
  Let $X$ and $Y$ be uniformly locally finite simplicial
  complexes. Suppose that they are boundedly homotopic in $\ell^2$ and
  $\ell^p$ norms for some $p \geq 2$. Then one has
  \begin{equation}
    \label{eq:58}
    \overline{d_k (\ell^2 X^k)}^{\ell^2} \subset
    d_k(\ell^p X^k )\,, 
  \end{equation}
  if and only if a similar inclusion holds on $Y$.
\end{prop}
\begin{proof}
  Suppose that $\overline{d_k (\ell^2 X^k)}^{\ell^2} \subset
  d_k(\ell^p X^k )$ and consider a sequence $\alpha_n = d'_k (\beta_n)
  \in d'_k(\ell^2 Y^k)$ that converges to $\alpha \in \overline{d_k
    (\ell^2 Y^k)}^{\ell^2}$ in $\ell^2$.

  Then $g_{k+1} \alpha_n = d_k(g_k \beta_n) \rightarrow g_{k+1} \alpha
  \in \overline{d_k (\ell^2 X^k)}^{\ell^2}$. Therefore there exists
  $\beta \in \ell^p X^k$ such that $g_{k+1} \alpha = d_k \beta $. Then
  taking $\ell^2$-limit in the sequence
  \begin{displaymath}
    f_{k+1} g_{k+1} \alpha_n = \alpha_n + d'_k h'_{k+1} \alpha_n + h'_{k+2}
    d'_{k+1} \alpha_n  = \alpha_n + d'_k h'_{k+1} \alpha_n
  \end{displaymath}
  gives
  \begin{displaymath}
    d'_k (f_k \beta) = f_{k+1} d_k \beta = \alpha + d'_k
    h'_{k+1}\alpha \,,
  \end{displaymath}
  and finally $\alpha \in d'_k(\ell^p Y^k)$ since $\ell^2 Y^k \subset
  \ell^p Y^k$ for $p \geq 2$.
\end{proof}

The inclusion \eqref{eq:58} we consider here is related to problems
studied in $\ell^{p,q}$ cohomology. We briefly recall this notion and
refer for instance to \cite{Goldshtein-Troyanov} for more details. If
$X$ is a simplicial complex as above, one considers the spaces
\begin{equation*}
  Z_q^k (X) = \ker d_k \cap
  \ell^q X^k \quad \mathrm{and} \quad B_{p,q}^k(X) =
  d_{k-1}(\ell^p X^k) \cap \ell^q X^k \,.
\end{equation*}
Then the $\ell^{p,q}$-cohomology of $X$ is defined by
\begin{displaymath}
  H_{p,q}^k (X) = Z_q^k (X) /
  B_{p,q}^k(X)\,. 
\end{displaymath}
Its reduced part is the Banach space
\begin{displaymath}
  \overline H_{p,q}^k (X) =
  Z_q^k (X) / \overline B_{p,q}^k(X) \,,
\end{displaymath} 
while its torsion part
\begin{displaymath}
  T_{p,q}^k(X) = \overline B_{p,q}^k(X)/
  B_{p,q}^k(X) \,
\end{displaymath}
is not a Banach space. These spaces fit into the exact sequence
\begin{displaymath}
  0 \rightarrow T_{p,q}^k(X) \rightarrow H_{p,q}^k
  (X)  \rightarrow \overline H_{p,q}^k (X)
  \rightarrow 0 \,.
\end{displaymath}
It is straightforward to check as above that, for $p\geq q$, these
spaces satisfy the same homotopical invariance property as in
Proposition~\ref{prop:invariance-torsion}. 

\begin{prop}
  \label{prop:invariance-cohomology}
  Let $X$ and $Y$ be uniformly locally finite simplicial
  complexes. Suppose that they are boundedly homotopic in $\ell^p$ and
  $\ell^q$ norms for $p\geq q$. Then the maps $f_k: \ell^* X^k
  \rightarrow \ell^* Y^k$ and $g_k : \ell^* Y^k \rightarrow \ell^*
  X^k$ induce reciprocal isomorphisms between the $\ell^{p,q}$
  cohomologies of $X$ and $Y$, as well as their reduced and torsion
  components.
\end{prop}
In this setting, the vanishing of the $\ell^{p,2}$-torsion
$T_{p,2}^{k+ 1}(X)$ is equivalent to the closeness of
$B^{k+1}_{p,2}(X) = d_k(\ell^p X^k) \cap \ell^2 X^{k+1}$ in $\ell^2
X^{k+1}$, i.e. to the inclusion
\begin{displaymath}
  \overline{d_k(\ell^p X^k) \cap \ell^2 X^{k+1}}^{\ell^2} \subset
  d_k(\ell^p X^k) \cap \ell^2 X^{k+1}  \,.
\end{displaymath}
This implies the weaker inclusion \eqref{eq:58}, but is stronger in
general unless the following holds
\begin{equation}
  \label{eq:59}
  d_k(\ell^p X^k) \cap \ell^2 X^{k+1} \subset \overline{d_k(\ell^2
    X^k)}^{\ell^2} .
\end{equation}
Now by Hodge decomposition in $\ell^2 X^{k+1}$, one has always
\begin{displaymath}
  d_k(\ell^p X^k) \cap \ell^2 X^{k+1} \subset \ker d_{k+1} \cap
  \ell^2 X^{k+1} = \overline{H}_{2}^{k+1}(X) \oplus^\bot \overline{d_k(\ell^2
    X^k)}^{\ell^2} .
\end{displaymath}
Hence \eqref{eq:59} holds if the reduced $\ell^2$-cohomology
$\overline{H}_{2}^{k+1}(X)$ vanishes, proving in that case the
equivalence of \eqref{eq:58} to the vanishing of the
$\ell^{p,2}$-torsion, and even to the identity
\begin{equation}
  \label{eq:60}
  B_{p,2}^{k+1} : = d_k (\ell^p X^k) \cap \ell^2 X^{k+1} = \overline{d_k(\ell^2
    X^k)}^{\ell^2} , 
\end{equation}
which is clearly closed in $\ell^2$.
\begin{cor}
  \label{cor:lp2}
  Let $K$ be a finite simplicial space and $\Gamma \rightarrow X
  \rightarrow K$ a covering. Suppose that the spectral distribution
  $F_{\Gamma,k}$ of $A=d_k^* d_k$ on $(\ker d_k)^\bot$ satisfies
  $F_{\Gamma,k}(\lambda) \leq C \lambda^{\alpha/2}$ for some $\alpha >
  2$. Suppose moreover that the reduced $\ell^2$-cohomology
  $\overline{H}_{2}^{k+1}(X) $ vanishes.

  Then \eqref{eq:60} and the vanishing of the $\ell^{p,2}$-torsion
  $T_{p,2}^{k+1}(X)$ hold for $1/p \leq 1/2 - 1/\alpha$.
\end{cor}

For instance, by \cite{Cheeger-Gromov}, infinite amenable groups have
vanishing reduced $\ell^2$-cohomology in all degrees.

\bibliographystyle{abbrv}


\def\cprime{$'$}

-----------------------------------------------------

\end{document}